\documentclass{article}
\usepackage{amsmath,amssymb,amsthm,amscd,dsfont}
\numberwithin{equation}{section} \allowdisplaybreaks
\begin{document}
\newtheorem{theorem}{Theorem}[section]
\newtheorem{defin}{Definition}[section]
\newtheorem{prop}{Proposition}[section]
\newtheorem{corol}{Corollary}[section]
\newtheorem{lemma}{Lemma}[section]
\newtheorem{rem}{Remark}[section]
\newtheorem{example}{Example}[section]
\title{Geometry on Big-Tangent Manifolds}
\author{{\small by}\vspace{2mm}\\Izu Vaisman}
\date{}
\maketitle
{\def\thefootnote{*}\footnotetext[1]%
{{\it 2000 Mathematics Subject Classification: 53C15, 53C80} .
}}
\begin{center} \begin{minipage}{10cm}
A{\footnotesize BSTRACT. Motivated by generalized geometry, we discuss differential geometric structures on the total space $\mathfrak{T}M$of the bundle $TM\oplus T^*M$, where $M$ is a differentiable manifold; $\mathfrak{T}M$ is called a big-tangent manifold. The vertical leaves of the bundle are para-Hermitian vector spaces. The big-tangent manifolds are endowed with canonical presymplectic, Poisson and $2$-nilpotent structures. We discuss lifting processes from $M$ to $\mathfrak{T}M$. From the point of view of the theory of $G$-structures, the structure of a big-tangent manifold is equivalent with a suitable triple $(P,Q,S)$, where $P$ is a regular bivector field, $Q$ is a $2$-contravariant symmetric tensor field of the same rank as $P$ and $S$ is a $2$-nilpotent $(1,1)$-tensor field. The integrability conditions include the annulation of the Schouten-Nijenhuis bracket $[P,P]$,  the annulation of the Nijenhuis tensor $\mathcal{N}_S$ and conditions that connect between the three tensor fields. We discuss horizontal bundles and associated linear connections with the Bott property. Then, we discuss metrics on the vertical bundle that are compatible with the para-Hermitian metric of the leaves. Together with a horizontal bundle, such metrics may be seen as a generalization of the double fields of string theory with the role of double fields over a manifold. We define a canonical connection and the action functional of such a field.}
\end{minipage}
\end{center} \vspace*{5mm}
\section{Introduction}
The geometry of the total space of a tangent bundle is of considerable interest in differential geometry and it was the subject of many publications. In the more recent subject of {\it generalized geometry} introduced by Hitchin
\cite{Ht1}, which is of interest in theoretical physics (e.g., \cite{Zab}), the tangent bundle $TM$ of the
$m$-dimensional, differentiable manifold $M$ is replaced by the {\it big tangent
bundle} $\mathbf{T}M=TM\oplus T^*M$. Accordingly, we investigate the geometry of the total space of a big tangent bundle, called a big-tangent manifold, which we will denote by $\mathfrak{T}M$. The name big tangent bundle will refer to either the bundle structure or its total space as indicated by the context.

We recall that the bundle $ \mathbf{T}M$ has a non degenerate, neutral metric and a non degenerate $2$-form defined by
\begin{equation}\label{g}
g(\mathcal{X},\mathcal{Y})=\frac{1}{2}(\alpha(Y)+\mu(X)),\; \omega(\mathcal{X},\mathcal{Y})=\frac{1}{2}(\alpha(Y)-\mu(X)),
\end{equation} and the Courant bracket
\begin{equation}\label{Cbr}
[\mathcal{X},\mathcal{Y}]=([X,Y],L_X\mu-L_Y\alpha +\frac{1}{2}d(\alpha(Y)-\mu(X)),\end{equation}
where calligraphic characters denote pairs
$\mathcal{X}=(X,\alpha),\mathcal{Y}=(Y,\mu)$ with $X,Y$ either vectors or vector fields and $\alpha,\mu$ either covectors or $1$-forms. The structure group of $(\mathbf{T}M,g)$ is $O(m,m)$ and generalized geometric structures are defined as reductions of this structure group to various subgroups.

In the paper, we use the typical notation of differential geometry, e.g., like in \cite{KN}, and the reader may refer to \cite{Mol} for the encountered foliation theory notions. For the evaluation of exterior and symmetric products we use Cartan's convention
$$\begin{array}{l}
\alpha\wedge\beta(X,Y)=\alpha(X)\beta(Y)-\alpha(Y)\beta(X),
\vspace*{2mm}\\
\alpha\odot\beta(X,Y)=\alpha(X)\beta(Y)+\alpha(Y)\beta(X),
\end{array}$$
etc., without a factor $1/2$ in the right hand side.
Everything will be $C^\infty$-smooth (possibly, except along the zero section of a vector bundle).

In Section 2 we show that the leaves of the vertical foliation of a big-tangent manifold are para-Hermitian vector spaces and that the manifold has canonical presymplectic, Poisson, $2$-nilpotent and generalized $2$-nilpotent structures. Then, we define the vertical and complete lifts of vector fields from $M$ to $\mathfrak{T}M$, which are similar to those from $M$ to $TM$, and new lifts from $TM$ and $T^*M$ to $\mathfrak{T}M$.

In Section 3 we give the interpretation of the structure of $\mathfrak{T}M$ as a $G$-structure, equivalent to a suitable triple $(P,Q,S)$, where $P$ is a regular bivector field, $Q$ is a $2$-contravariant symmetric tensor field of the same rank as $P$ and $S$ is a $(1,1)$-tensor field with $S^2=0$. This leads to {\it almost big-tangent manifolds} and their integrable case the {\it big-tangent manifolds}. We establish the corresponding integrability conditions, which consist of the annulation of the Schouten-Nijenhuis bracket $[P,P]$, of the Nijenhuis tensor $\mathcal{N}_S$ and of the Lie derivatives of $S$ with respect to $P$-Hamiltonian vector fields.

In Section 4 we discuss horizontal bundles. We lift horizontal bundles of the usual tangent manifold to the big tangent manifold, particularly, those defined by regular Lagrangians. Then, we consider the linear connections that have the Bott property known from foliation theory and we show the main properties of the torsion and curvature of these connections.

In Section 5 we discuss metrics and vertical metrics on a big-tangent manifold. We review the canonical connection of foliation theory for the vertical foliation and its curvature properties. Vertical metrics that are compatible with the para-Hermitian metric of the leaves, together with a horizontal bundle, may be seen as a generalization of the notion of double field studied in string theory. We transfer to this case the para-Hermitian construction of an invariant action of the field given in \cite{VT}.
\section{The big tangent manifold}
The big tangent manifold $ \mathfrak{T}M$ has dimension $3m$, $m=dim\,M$.
The points of $ \mathfrak{T}M$ are triples $(x\in M,y\in T_xM,z\in T^*M)$ and one has natural local coordinates $(x^i,y^i,z_i)$, where $i=1,...,m$, $(x^i)$ are local coordinates on $M$, $(y^i)$ are vector coordinates and $(z_i)$ are covector coordinates. The corresponding coordinate transformations are:
\begin{equation}\label{coordtr} \tilde{x}^i=\tilde{x}^i(x^j),
\,\tilde{y}^i=\frac{\partial\tilde{x}^i}{\partial x^j}y^j,\,
\tilde{z}_i=\frac{\partial{x}^j}{\partial \tilde{x}^i}z_j.
\end{equation}
Notice the existence of the global function $ev(x,y,z)=z(y)=z_iy^i$, $ev\in C^\infty(\mathfrak{T}M)$, called the {\it evaluation function} (we use the Einstein summation convention overall).

On $\mathfrak{T}M$, a tangent vector $\mathfrak{X}$ and a $1$-form $\mathfrak{a}$ have the coordinate expressions
\begin{equation}\label{vect} \mathfrak{X}=\xi^i\frac{\partial}{\partial{x}^i}
+\eta^i\frac{\partial}{\partial{y}^i}
+\zeta_i\frac{\partial}{\partial{z}_i},\end{equation}
\begin{equation}\label{covect} \mathfrak{a}=\alpha_idx^i+\beta_idy^i+\gamma^idz_i,\end{equation}
and a coordinate transformation (\ref{coordtr}) implies the following change of vector and covector coordinates:
\begin{equation}\label{vectortr}
\tilde{\xi}^i=\frac{\partial\tilde{x}^i}{\partial x^j}\xi^j,\,
\tilde{\eta}^i=\frac{\partial\tilde{y}^i}{\partial x^j}\xi^j
+\frac{\partial\tilde{x}^i}{\partial x^j}\eta^j,\,
\tilde{\zeta}_i=\frac{\partial\tilde{z}_i}{\partial x^j}\xi^j
+\frac{\partial{x}^j}{\partial \tilde{x}^i}\zeta_j,
\end{equation}
$$ \tilde{\alpha}_i=\frac{\partial{x}^j}{\partial \tilde{x}^i}\alpha_j+\frac{\partial{y}^j}{\partial \tilde{x}^i}\beta_j +\frac{\partial{z}_j}{\partial \tilde{x}^i}\gamma^j,\, \tilde{\beta}_i=\frac{\partial{x}^j}{\partial \tilde{x}^i}\beta_j,\, \tilde{\gamma}^i=\frac{\partial\tilde{x}^i}{\partial {x}^j}\gamma^j.$$

The manifold $\mathfrak{T}M$ has the projections
$$ p:\mathfrak{T}M\rightarrow M,\,
p_1:\mathfrak{T}M\rightarrow TM,\,p_2:\mathfrak{T}M\rightarrow T^*M$$ on $M$ and on the total spaces of the tangent and cotangent bundle, respectively.

For any fiber bundle, the tangent bundle of the fibers is called vertical and, usually, denoted by $\mathcal{V}$. On $\mathfrak{T}M$, the vertical bundle has the decomposition
\begin{equation}\label{descV} \mathcal{V}\mathfrak{T}M=
\mathcal{V}_1\oplus\mathcal{V}_2,
\end{equation} where
$$\mathcal{V}_1=p_1^{-1}(
\mathcal{V}TM),\,\mathcal{V}_2=p_2^{-1}(\mathcal{V}T^*M)$$
and have the local bases $(\partial/\partial y^i)$, $(\partial/\partial z_i)$, respectively.
We also have the isomorphisms \begin{equation}\label{fiso} \mathcal{V}_1\approx p^{-1}(TM),\,\mathcal{V}_2\approx p^{-1}(T^*M),\,
\mathcal{V}\approx p^{-1}(\mathbf{T}M).\end{equation}

The subbundles $\mathcal{V}_1,\mathcal{V}_2$ are the foliations of $\mathfrak{T}M$ by the fibers of $p_2,p_1$, respectively, and $\mathfrak{T}M$ has a multi-foliate structure \cite{{KS},{V70}} that consists of the diagram of foliations
\begin{equation}\label{diagram}
\{0\}=\mathcal{V}_1\cap\mathcal{V}_2\begin{array}{c}
\nearrow\mathcal{V}_1
\searrow\vspace*{2mm}\\ \searrow \mathcal{V}_2\nearrow
\end{array}\mathcal{V}_1\oplus\mathcal{V}_2
=\mathcal{V}\subset T\mathfrak{T}M,\end{equation}
where the arrows are inclusions; it will be called the {\it vertical multi-foliation}.

We denote by $\chi(\mathfrak{T}M),\chi_v(\mathfrak{T}M), \chi_a(\mathfrak{T}M)$ $(a=1,2)$ the spaces of all vector fields and of vector fields in $\mathcal{V},\mathcal{V}_1,\mathcal{V}_2$, respectively. The fiber-wise infinitesimal homotheties, called {\it Euler vector fields} are given by
\begin{equation}\label{EVF}
\mathcal{E}_1=y^i\frac{\partial}{\partial{y}^i}\in \chi_1(\mathfrak{T}M),\,
\mathcal{E}_2=z_i\frac{\partial}{\partial{z}_i}\in \chi_2(\mathfrak{T}M),\,
\mathcal{E}=\mathcal{E}_1+\mathcal{E}_2\in \chi_v(\mathfrak{T}M).\end{equation}
On $\mathfrak{T}M\setminus\{0\}$, where $\{0\}$ denotes the zero section, $\mathcal{E}\neq0$
and $\mathcal{E}$ is a transversal vector field of the codimension-one foliation
defined by $ev=const.$
\begin{prop}\label{paraHpeV} The vertical leaves of $\mathfrak{T}M$ have a natural structure of para-Hermitian vector spaces.\end{prop}
\begin{proof}  We refer to \cite{CFG} for para-Hermitian geometry.
The last isomorphism (\ref{fiso}) transfers the metric $g$ and $2$-form $\omega$ given by (\ref{g}) to a metric and a $2$-form on $\mathcal{V}$, also denoted by $g,\omega$, given by
$$ g(\mathfrak{X},\mathfrak{X}')= \frac{1}{2}(\zeta_i\eta^{'i}+\zeta'_i\eta^{i}),\;
\omega(\mathfrak{X},\mathfrak{X}')= \frac{1}{2}(\zeta_i\eta^{'i}-\zeta'_i\eta^{i})\;
(\mathfrak{X},\mathfrak{X}'\in\mathcal{V}),$$
where the coordinates are like in (\ref{vect})
(the coordinate characterization of $\mathcal{V}$ is $\xi^i=0$).
The musical isomorphism $\flat_{g}$ yields isomorphisms $\mathcal{V}_2\approx\mathcal{V}_1^*, \mathcal{V}_1\approx\mathcal{V}_2^*$. On the other hand,
(\ref{descV}) defines a product structure $F_{\mathcal{V}}$ on the leaves of $\mathcal{V}$, with $\mathcal{V}_1,\mathcal{V}_2$ as $\pm1$-eigenbundles and it is easy to see that $(g,F_{\mathcal{V}})$ is a para-Hermitian structure on the leaves with $\omega$ as fundamental form. Since the leaves of $\mathcal{V}$ are vector spaces, we are done.\end{proof}
\begin{prop}\label{structuri} The big tangent manifold $\mathfrak{T}M$ is endowed with canonical tensor fields $\varpi,P,Q,S$, where $\varpi$ is a presymplectic form of rank $2m$, $P$ is a Poisson bivector field, $Q$ is a symmetric $2$-contravariant tensor field and $S$ is a $2$-nilpotent structure. Furthermore, the triple $(S,P,Q)$ has the following properties:\\ 1)\hspace{5mm} $rank\,S=m,\;
ker\,S=im\,\sharp_P=im\,\sharp_Q$,\\ 2)\hspace{5mm} $\sharp_P\circ\flat_Q =\sharp_Q\circ\flat_P,\;	 \sharp_Q\circ\flat_P\circ S=-S$,\\
where flats are the inverses of the isomorphisms onto the image of sharps.
\end{prop}
\begin{proof} We take $\varpi$ to be the pullback of the canonical symplectic form of $T^*M$ by $p_2$, i.e., locally,
$$\varpi=d\lambda=-dx^i\wedge dz_i,\; \lambda=z_idx^i;$$ $\varpi$ is closed and has rank $2m$.

The inverse of $2\omega$ along the para-Hermitian leaves of $\mathcal{V}$ yields the bivector field \begin{equation}\label{P} P=\frac{\partial}{\partial y^i}
\wedge\frac{\partial}{\partial z_i},\end{equation} which is a regular Poisson structure of rank $2m$. The symplectic leaves of $P$ are the leaves of the vertical foliation $\mathcal{V}$.

The inverse of $2g$ along the leaves of $\mathcal{V}$ yields the twice contravariant tensor field of rank $2m$
\begin{equation}\label{G} Q=\frac{\partial}{\partial y^i}
\odot\frac{\partial}{\partial z_i}.\end{equation}

Finally, we define
$$S=\sharp_P\circ\flat_\varpi=dx^i\otimes\frac{\partial}{\partial y^i}.$$
On the vector (\ref{vect}) we get
\begin{equation}\label{Snou} S\mathfrak{X}=\xi^i\frac{\partial}{\partial y^i},
\end{equation} hence, $rank\,S=m$. We also get
\begin{equation}\label{reltens} S^2=0,\, S\circ\sharp_P=0,\,\flat_\omega\circ S=0.\end{equation}
Recall that $S$ is a $2$-nilpotent structure if $S^2=0$, $rank\,S=const.$ and the integrability condition
\begin{equation}\label{Nij} \mathcal{N}_{S}(\mathfrak{X},\mathfrak{X}')=[S\mathfrak{X}, S\mathfrak{X}']-S([S\mathfrak{X},\mathfrak{X}']
+[\mathfrak{X},S\mathfrak{X}'])=0\end{equation}
holds, which happens in our case.

Properties 1) and 2) follow from the local expressions of the tensor fields. In properties 2) we refer to the isomorphisms $$\sharp_P:T^*\mathfrak{T}M/ker\,\sharp_P\approx
im\,\sharp_P,\;\sharp_Q:T^*\mathfrak{T}M/ker\,\sharp_Q\approx
im\,\sharp_Q,$$ the flats are the inverses of these isomorphisms and the composition makes sense because of the equality of the kernels.
\end{proof}
\begin{rem}\label{obsstr} {\rm	Properties 1) are equivalent with\\
$1'$)\hspace{5mm} $rank\,P=rank\,Q=2m,\;ker\,\sharp_P=ker\,\sharp_Q=im\,
\hspace{1pt}^t\hspace{-1pt}S$.\\
Notice the following relations $$L_{\mathcal{E}}\lambda=\lambda, \; L_{\mathcal{E}}\varpi=\varpi,\;	 \sharp_P\lambda=0,\; \sharp_Gd(ev)=\mathcal{E},\; L_{\mathcal{E}}P=-2P,\; L_{\mathcal{E}}Q=-2Q$$
and the $2$-contravariant tensor field
\begin{equation}\label{U} U=\frac{1}{2}(Q+P)=\frac{\partial}{\partial y^i}
\otimes\frac{\partial}{\partial z_i}.\end{equation}}\end{rem}
\begin{prop}\label{strgen} A big tangent manifold $\mathfrak{T}M$ is endowed with two canonical, generalized, $2$-nilpotent structures.\end{prop}
\begin{proof} We refer to \cite{VD} for the notion of a generalized $2$-nilpotent structure. The structures required by the proposition are given by the endomorphisms
$S_P,S_\varpi:\mathbf{T}\mathfrak{T}M\rightarrow \mathbf{T}\mathfrak{T}M$ with the matrix representation
$$ S_P=\left(\begin{array}{cc} S&\sharp_P\vspace{2mm}\\
0&-^t\hspace{-1pt}S\end{array}\right),\;
S_\varpi=\left(\begin{array}{cc} S&0\vspace{2mm}\\
\flat_\varpi&-^t\hspace{-1pt}S\end{array}\right),$$
where the index $t$ denotes transposition.
Properties (\ref{reltens}) show that
these are generalized almost $2$-nilpotent structures. The structures are integrable since the Courant-Nijenhuis tensors of $S_P,S_\varpi$, which are defined by a formula similar to (\ref{Nij}) where the arguments are cross sections of $\mathbf{T}\mathfrak{T}M$, the brackets are Courant brackets and $S$ is replaced by $S_P,S_\varpi$, vanish (it suffices to check on the natural bases).\end{proof}

Now, we will discuss lifting procedures that extend those used for usual tangent bundles. A cross section $$(X=\xi^i\frac{\partial}{\partial x^i},\alpha=\alpha_idx^i)\in \Gamma\mathbf{T}M$$ has a {\it generalized moment} defined by
\begin{equation}\label{linfunct}
l_{(X,\alpha)}(x,y,z)=\alpha(y)+z(X)=\alpha_iy^i+z_i\xi^i,
\end{equation}
which together with the functions $p^*f$ ($f\in C^\infty(M)$) functionally generate $C^\infty(\mathfrak{T}M)$.
\begin{prop}\label{propliftv} For any point $x\in M$, there exists a canonical isomorphism $v:\mathbf{T}_xM\rightarrow\mathcal{V}_{(x,y,z)}\subset T_{(x,y,z)}(\mathfrak{T}M)$.\end{prop}
\begin{proof}
The vector $v(X,\alpha)$, also denoted $(X^v,\alpha^v)$, will be called the {\it vertical lift}, and it is defined by the directional derivatives $$(X^v,\alpha^v)(p^*f)=0,\,(X^v,\alpha^v)(l_{(Y,\beta)})=
\alpha(Y)+\beta(X).$$ In local coordinates we get
$$ (X^v,\alpha^v)=\xi^i\frac{\partial}{\partial y^i}+\alpha_i \frac{\partial}{\partial z_i}.$$
Obviously, the vertical lift is an isomorphism onto the vertical space.\end{proof}
\begin{rem}\label{obsv} {\rm 1. We will denote $(X^v,0^v)=X^v,(0^v,\alpha^v)=\alpha^v$. 2. The vertical lift satisfies the property
$$ [(X^v,\alpha^v),(Y^v,\beta^v)]=0.$$ 3. Using the vertical lift we get $S\mathfrak{X}=(p_*\mathfrak{X})^v$.
4. The vertical lift leads to push-forward homomorphisms
$$ q':T_{(x,y,z)}^*\mathfrak{T}M\rightarrow T_xM,\,q'':T_{(x,y,z)}^*\mathfrak{T}M\rightarrow T_x^*M,$$
where
$$ \alpha(q'\mathfrak{a})=
\mathfrak{a}(\alpha^v),\, q''\mathfrak{a}(X)=\mathfrak{a}(X^v),$$
$X\in T_xM,\alpha\in T^*_xM$. For $\mathfrak{a}$ given by (\ref{covect}), we have
$$ q'\mathfrak{a}=\gamma^i
\frac{\partial}{\partial x^i},\,
q''\mathfrak{a}=\beta_idx^i,$$ and the following relations hold
$$\sharp_G\mathfrak{a}=(q'\mathfrak{a})^v+(q''\mathfrak{a})^v,
(q''\mathfrak{a})^v=U(\mathfrak{a}),$$
where $U$ is (\ref{U}) seen as the homomorphism $\mathfrak{a}\mapsto i(\mathfrak{a})U$ (contraction on the first index of $U$). 5. The vertical lift extends to tensors like in the case of the usual tangent bundle, but, here, it always leads to completely contravariant tensors.}\end{rem}
\begin{prop}\label{propliftc} For any manifold $M$, there exists a canonical injection $c:\chi(M)\rightarrow\chi(\mathfrak{T}M)$, which is compatible with the flows.\end{prop}
\begin{proof} Like in the case of $TM$, we define the required injection $c(X)=X^c$, called {\it complete lift}, by the directional derivatives
\begin{equation}\label{liftc2} X^c(p^*f)=p^*(Xf),\,
X^c(l_{df})=l_{d(Xf)},\,X^c(l_Y)=l_{[X,Y]},\end{equation}
where $f\in C^\infty(M)$ and $X,Y\in\chi(M)$.
By means of local coordinates, if $X=\xi^i(\partial/\partial x^i)$, then,
\begin{equation}\label{liftc} X^c=\xi^i\frac{\partial}{\partial x^i}+y^j\frac{\partial\xi^i}{\partial x^j}\frac{\partial}{\partial y^i}-z_j\frac{\partial\xi^j}{\partial x^i}\frac{\partial}{\partial z_i}.\end{equation}
This formula shows that $c$ is an injection and that $X^c$ is projectable by the three projections $p,p_1,p_2$.

A diffeomorphism $\Phi$ of $M$ lifts to $\mathfrak{T}M$ acting by $\Phi_*$ on $y$ and by $\Phi^{-1*}$ on $z$, and a left action of a group $G$ on $M$ lifts to a left action on $\mathfrak{T}M$. Calculations similar to those for the tangent bundle show that the flow $exp(tX)$ lifts to the flow $exp(tX^c)$, where $X^c$ is given by (\ref{liftc}), which is the meaning of flow compatibility.\end{proof}
\begin{prop}\label{proplc} The complete lift has the following properties:
\begin{equation}\label{propc} \begin{array}{c}
g(X^c,\alpha^v)=\frac{1}{2}
(\alpha(X))^v,\,[X^c,Y^v]=[X,Y]^v,\,[X^c,Y^c]=[X,Y]^c,\vspace*{2mm}\\
(fX)^c=f^vX^c+l_{df}X^v-l_X(df)^v.\end{array}\end{equation}
\end{prop}
\begin{proof} Use the formulas (\ref{liftc2}) and (\ref{liftc}).\end{proof}
\begin{rem}\label{obslc} {\rm The last relation (\ref{propc}) tells us that the complete lift is not $C^\infty(M)$-linear, therefore, it does not extend well to higher order contravariant tensor fields. On the other hand, if needed, we may define complete lifts of covariant tensor fields as pullbacks of the complete lifts to $TM$ by the projection $p_1$.}\end{rem}

The defining formulas (\ref{liftc2}) of the complete lift suggest natural lifts with respect to the projections $p_1,p_2$. Let
$$\mathcal{X}=\xi^i(x)\frac{\partial}{\partial x^i}+\eta^i(x,y)\frac{\partial}{\partial y^i}$$
be a vector field on the total space of the bundle $TM$ that is projectable with
respect to the vertical foliation. The space of functions $C^\infty(\mathfrak{T}M)$ may be seen as locally, functionally spanned by functions of the type $p_1^*\varphi,l_{\mathcal{Y}}$ where
$\varphi\in C^\infty(TM)$, $\mathcal{Y}=\lambda^k(\partial/\partial y^k)$ is a vertical vector field on $TM$ and $l_{\mathcal{Y}}=z_k\lambda^k$.
\begin{defin}\label{extended1} {\rm The {\it extended lift} $\mathcal{X}^e$ of $\mathcal{X}$ to $\mathfrak{T}M$ is the vector field defined by the directional derivatives
 \begin{equation}\label{extend1}\mathcal{X}^e(p_1^*\varphi)=
 p_1^*(\mathcal{X}\varphi),\; \mathcal{X}^e(l_{\mathcal{Y}})=l_{[\mathcal{X},\mathcal{Y}]}
 \hspace{2mm}(\varphi\in C^\infty(TM)).
\end{equation}}\end{defin}

Conditions (\ref{extend1}) are correct since $\mathcal{X}$ projectable implies that the bracket $[\mathcal{X},\mathcal{Y}]$ is vertical.
From (\ref{extend1}) we get the local expression
\begin{equation}\label{extend2}\mathcal{X}^e= \xi^i\frac{\partial}{\partial x^i}+\eta^i\frac{\partial}{\partial y^i}-z_j\frac{\partial\eta^j}{\partial y^i}
\frac{\partial}{\partial z_i}.\end{equation}

Similarly, if
$$\mathcal{U}=\xi^i(x)\frac{\partial}{\partial x^i}+\zeta_i(x,z)\frac{\partial}{\partial z_i}$$ is a vertically projectable vector field on the manifold $T^*M$ and if we use functions $l_{\mathcal{Z}}=y^i\theta_i$ where $\mathcal{Z}=\theta_i(\partial/\partial z_i)$, we define
 \begin{defin}\label{extended2} {\rm The {\it extended lift} $\mathcal{U}^e$ of $\mathcal{U}$ to $\mathfrak{T}M$ is the vector field defined by the directional derivatives
$$ \mathcal{U}^e(p_2^*\varphi)=
 p_2^*(\mathcal{U}\varphi),\; \mathcal{U}^e(l_{\mathcal{Z}})=l_{[\mathcal{U},\mathcal{Z}]}
 \hspace{2mm}(\varphi\in C^\infty(T^*M)).$$}\end{defin}
The extended lift to $\mathfrak{T}M$ is given by
\begin{equation}\label{extend3} \mathcal{U}^e=\xi^i\frac{\partial}{\partial x^i}-z_j\frac{\partial\zeta_j}{\partial z_i}\frac{\partial}{\partial y^i}+\zeta_i\frac{\partial}{\partial z_i}.\end{equation}
\begin{rem}\label{obsce} {\rm
The extended lift of the complete lift of a vector field of $M$ to $TM$ is the complete lift to $\mathfrak{T}M$.}\end{rem}
\begin{rem}\label{obseliftgen} {\rm The extended lift may be extended to a lift of a projectable cross section of the bundle $p_1^{-1}(TM)$, respectively $p_1^{-1}(T^*M)$, to a vector field on $\mathfrak{T}M$. The vertical bundle of the pullback is the pullback of the vertical bundle of $TM,T^*M$, respectively, the Lie brackets are calculated as if the coordinates $z_i$, respectively $y^i$, were parameters, and a cross section of the pullback is {\it projectable} if the Lie brackets with vertical cross sections are again vertical. The local expressions (\ref{extend2}), (\ref{extend3}) remain valid, but, in (\ref{extend2}), $\xi^i=\xi^i(x,z),\eta^i=\eta^i(x,y,z)$ and, in (\ref{extend3}), $\xi^i=\xi^i(x,y),\zeta_i=\zeta_i(x,y,z)$.} \end{rem}
\section{The $G$-structure theory framework}
$G$-structures are reductions of the structure group of the tangent bundle of a manifold to a subgroup $G\subseteq Gl(n,\mathds{R})$, and $G$-structure theory is a general framework to study geometry on manifolds, e.g., \cite{DB}.

The coordinate transformations (\ref{coordtr}) may be interpreted as telling us that $\mathfrak{T}M$ has an integrable $G$-structure where $G$ is the following linear subgroup
$$G=Bt(3m,\mathds{R})=\left( \begin{array}{ccc} A&B&C\vspace*{2mm}\\ 0&A&0\vspace*{2mm}\\ 0&0&\hspace*{1pt}^t\hspace*{-1pt}A^{-1}\end{array}\right),
\hspace{2mm}A\in Gl(m,\mathds{R}).$$
Accordingly, a	$3m$-dimensional manifold $N$ with a $Bt(3m,\mathds{R})$-structure is to be seen as an {\it almost big-tangent manifold}. A characterization of such structures is given in the following proposition.
\begin{prop}\label{Gstr1} Let $N$ be a $3m$-dimensional differentiable manifold. An almost big-tangent structure on $N$ is equivalent with a triple of tensor fields $(S,P,Q)$, where $S$ is of type $(1,1)$, $P$ and $Q$ are of type $(2,0)$, $P$ is skew symmetric and $Q$ is symmetric, and these tensor fields satisfy the properties 1), 2) of Proposition \ref{structuri}.\end{prop}
\begin{proof}
Seen as a $G$-structure, the almost big-tangent structure is a bundle of frames (seen as one-line matrices of vectors) with frame changes in
 $Bt(3m,\mathds{R})$ and such a frame is of the form $(a_i,b_i,c^i)$ $(i=1,...,m)$. If we define
 $$Sa_i=b_i,\,Sb_i=0,\,Sc^i=0,\;P=b_i\wedge c^i,\;Q=b_i\odot c^i,$$ we get global tensor fields that have the required properties.

Conversely, if we have $(S,P,Q)$, we may look at tangent frames $(a_i,b_i,c^i)$ such that $b_i\in im\,S$ and $Sa_i=b_i$, then, define $c^i$ as follows. Take the non degenerate metric $g$ defined on the subbundle $\mathcal{V}=im\,\sharp_Q\subseteq TN$ by $$g(\mathfrak{Y},\mathfrak{Y}')=<\flat_Q\mathfrak{Y},\mathfrak{Y}'>.$$
The metric $g$ is para-Hermitian on $\mathcal{V}$ with the paracomplex structure tensor $\phi=\sharp_Q\circ\flat_P$ and $\mathcal{V}$ is the $+1$-eigenbundle of $\phi$.
Accordingly, $g(b_i,b_j)=0$, and there exist vectors $\tilde{c}^i$ such that
\begin{equation}\label{auxGstr}
g(\tilde{c}^i,\tilde{c}^j)=0,\,g(b_i,\tilde{c}^j)=\delta_i^j.
\end{equation}
Finally, define $c^i=(1/2)(\tilde{c}^i-\phi \tilde{c}^i)$. These vectors still satisfy (\ref{auxGstr}) and yield a basis of the $-1$-eigenbundle of $\phi$, while
$b_i$ is a basis of the $+1$-eigenbundle.  Now, we see that, if $b_i$ change by $A\in Gl(m,\mathds{R})$, $c^i$ change by $\hspace{1pt^t}\hspace{-1pt}A^{-1}$ and the whole basis changes by an element of  $Bt(3m,\mathds{R})$.
\end{proof}
\begin{defin}\label{defintegrab} {\rm An almost big-tangent structure $(S,P,Q)$ is {\it quasi-integrable} and called a {\it quasi-big-tangent structure} if there exists a {\it canonical atlas} of local coordinates $(x^i,y^i,z_i)$ $(i=1,...,m)$ such that the coordinate expressions of the tensors $S$ and $P$ are (\ref{Snou}) and (\ref{P}), respectively, and the structure is called {\it integrable} or {\it big-tangent} if it is quasi-integrable and $Q$ has the expression (\ref{G}).}\end{defin}
\begin{prop}\label{atlasintegr} The Jacobian matrix of a coordinate transformation between maps of the canonical atlas of a quasi-big-tangent structure $(S,P,Q)$ is of the form
$$
\left( \begin{array}{ccc} \frac{\partial\tilde{x}^i}{\partial x^j}
& \frac{\partial\tilde{y}^i}{\partial x^j}&
\frac{\partial\tilde{z}_i}{\partial x^j}\vspace*{2mm}\\ 0&
\frac{\partial\tilde{x}^i}{\partial x^j}&0\vspace*{2mm}\\ 0&\frac{\partial\tilde{y}^i}{\partial z_j}&\frac{\partial x^j}{\partial\tilde{x}^i}\end{array}\right)
$$
and $Q$ is of the form
$$Q=Q^{ij}\frac{\partial}{\partial y^i}\odot\frac{\partial}{\partial y^j}
+\frac{\partial}{\partial y^i}\odot\frac{\partial}{\partial z_i}\;\;(Q^{ij}=Q^{ji}).$$ If the structure is integrable, one also has $\partial\tilde{y}_i/\partial z_j=0$.\end{prop}
\begin{proof} Let $(x^i,y^i,z_i)$, $(\tilde{x}^i,\tilde{y}^i,\tilde{z}_i)$ be two maps of the canonical atlas. Since $im\,S=span\{\partial/\partial y^i\}=span\{\partial/\partial \tilde{y}^i\}$, we get $\partial \tilde{x}^i/\partial y^j=\partial \tilde{z}_i/\partial y^j=0$. Furthermore, from $S(\partial/\partial x^i)=\partial/\partial y^i,S(\partial/\partial \tilde{x}^i)=\partial/\partial \tilde{y}^i,S(\partial/\partial z_i)=
S(\partial/\partial \tilde{z}_i)=0$, we get $\partial\tilde{x}^i/\partial z_j=0$ and
$\partial \tilde{y}^i/\partial y^j=\partial \tilde{x}^i/\partial x^j$. Then,
$$P=\frac{\partial}{\partial y^i}\wedge\frac{\partial}{\partial z_i}=\frac{\partial}{\partial \tilde{y}^i}\wedge\frac{\partial}{\partial \tilde{z}_i}$$
implies $\partial\tilde{z}_i/\partial z_j=\partial x^j/\partial\tilde{x}^i$ and the Jacobian matrix has the required form. Furthermore, if we take a general, local expression of $Q$ and ask properties 1), 2) of Proposition \ref{structuri} to hold, we get the required expression of $Q$. The last assertion of the proposition holds because $Q$ has the canonical expression in both maps iff $\partial\tilde{y}^i/\partial z_j=0$.
\end{proof}

We recall that, separately, $P$ is integrable if the Schouten-Nijenhuis bracket is $[P,P]=0$, which is equivalent with the fact that $P$ defines a Poisson bracket. If this happens, $im\,\sharp_P$ is an integrable, generalized distribution with symplectic leaves. The converse is also true if the leaf-wise Hamiltonian fields of differentiable functions are differentiable on the manifold. In particular, if $rank\,P=const.$, $P$ is integrable iff $im\,P$ is a foliation and $P$ has the expression (\ref{P}) \cite{VPs}.

On the other hand, the almost $2$-nilpotent structure $S$, $S^2=0$, $rank\,S=const.$ is said to be integrable if the Nijenhuis tensor is $\mathcal{N}_S=0$. This situation is characterized by
\begin{prop}\label{integrS} A 2-nilpotent structure $S$ of rank $p$
on a manifold $N^{2q+p}$ is integrable iff $E=im\,S$ is a foliation, and there are
local coordinates $(t^h,x^u,y^u)$, ($h=1,...,p$, $u=1,...,q$) and independent, $E$-projectable, local
vector fields $T_h$ such that
\begin{equation}\label{Scanonic} S\frac{\partial}{\partial x^u}=\frac{\partial}{\partial{y}^u},\; S\frac{\partial}{\partial y^u}=0,\;ST_h=0.\end{equation}\end{prop}
\begin{proof} Formula (\ref{Scanonic}) shows that $E$ is a foliation. Then, $E$-projectability of $T_h$ means that $E$ is preserved by the infinitesimal transformations $T_h$ and we get $\mathcal{N}_S=0$. Conversely, $\mathcal{N}_S=0$ implies the integrability of $E$. Put $E'=ker\,S$. Since $im\,S\subseteq E'$
and $\mathcal{N}_S(U,Z)=-S[SU,Z]$ whenever $Z\in E'$, the
annulation of $\mathcal{N}_S$ gives $[SU,Z]\in\Gamma E'$. Hence, if $S$ is integrable, $E'$ is an $E$-projectable distribution and, following \cite{V3}, $TN$ has local bases of the following form
$$\frac{\partial}{\partial s^u},\; T_h=\frac{\partial}{\partial t^h}+\lambda_h^u\frac{\partial}{\partial s^u},
\;X_u=\frac{\partial}{\partial x^u}+\mu_u^v\frac{\partial}{\partial s^v},$$ where $(t^h,x^u,s^u)$ are local coordinates on $M$, $(\partial/\partial s^u)$ is a basis of $E$ and $(\partial/\partial s^u,T_h)$ is a basis of $E'$. Then, $SX_u=S(\partial/\partial x^u)$ and $\mathcal{N}_S=0$ gives
$[S(\partial/\partial x^u),S(\partial/\partial x^v)]=0$. This commutation property of vector fields tangent to the leaves of $E$ allows us to change the coordinates $s^u$ along the leaves of $E$ by new coordinates $y^u$ such that $S(\partial/\partial x^u)=\partial/\partial y^u$.
\end{proof}

Now, we come back to the integrability of a big-tangent structure and prove the following result.
\begin{prop}\label{propSP} The almost big-tangent structure $(S,P,Q)$ is quasi-integrable iff the structures $S$ and $P$ are integrable and the $P$-Hamiltonian vector fields preserve $S$, i.e.,
\begin{equation}\label{SsiP}\mathcal{N}_S=0,\,[P,P]=0,\, L_{\sharp_Pdf}S=0\hspace{2mm}(f\in C^\infty(N)). \end{equation} The same structure is integrable iff it satisfies (\ref{SsiP}) and there exists a $P$-Lagrangian, $Q$-isotropic, integrable distribution $\Delta$ such that $ker\,S=im\,S\oplus \Delta$.
\end{prop}
\begin{proof} By definition, quasi-integrability implies (\ref{SsiP}). Conversely assume that (\ref{SsiP}) holds. Properties 1), 2) of Proposition \ref{structuri} show that $im\,S$ is a Lagrangian subfoliation of the symplectic foliation of $P$, where the tangent spaces of the leaves are $ker\,S$. Then, a well known result of symplectic geometry (e..g., the Carath\'eodory-Jacobi-Lie theorem17.2 in \cite{LM}) tells us that there exists an atlas with local coordinates $(x^i,q^i,p_i)$ such that
$$ im\,S=span\{\frac{\partial}{\partial q^i}\},\; P=
\frac{\partial}{\partial q^i}\wedge\frac{\partial}{\partial p_i},\; S\frac{\partial}{\partial q^i}=0,\;S\frac{\partial}{\partial p_i}=0.$$
Accordingly, there are local, independent, vector fields $X_i=\xi_i^j(\partial/\partial x^j)$ such that $SX_i=\partial/\partial q^i$ and the last condition (\ref{SsiP}) implies $\xi_i^j=\xi_i^j(x)$. Then, we have $S(\partial/\partial x^i)=\eta_i^j(x)(\partial/\partial q^j)$, where $\xi^i_j\eta^j_k=\delta^i_k$, and $\mathcal{N}_S=0$ implies that the vector fields $S(\partial/\partial x^i)$ commute. Accordingly, there exist new local coordinates $y^i=y^i(x,q)$ along the leaves of $im\,S$ such that $S(\partial/\partial x^i)=\partial/\partial y^i$. If we extend our change of coordinates by the transformation $z_i=\eta_i^j(x)p_j$, we still have $P=(\partial/\partial y^i)\wedge(\partial/\partial z_i)$ and we see that the structure is quasi-integrable.

For the integrable case, what we meant by $P$-Lagrangian and $Q$-isotropic is
$$<\flat_PZ_1,Z_2>=0,\;\;<\flat_QZ_1,Z_2>=0,\;\; \forall Z_1,Z_2\in \Delta.$$
If the structure is integrable, we may take $\Delta=span\{\partial/\partial z_i\}$. Conversely, by the conditions assumed for $\Delta$, $x^i=const.$ along the leaves tangent to $\Delta$ and we may use the Carath\'eodory-Jacobi-Lie theorem along the symplectic leaves of $P$ to get coordinates $y^i,z_i$ on them, such that $P$ still has the canonical form, while $z_i$ are coordinates along the leaves of $\Delta$. Moreover, like in the ending argument of the quasi-integrability part of the proof, we can arrange all the coordinates such that both $P$ and $S$ preserve the canonical form. Finally, if we ask $Q$-isotropy of $span\{\partial/\partial z_i\}$ for $Q$ given by Proposition \ref{atlasintegr}, we get $Q^{ij}=0$ and we are done.
\end{proof}
\begin{rem}\label{obsafin2} {\rm
The coordinate transformations of the canonical atlas of an integrable structure $(S,P,Q)$ are not exactly (\ref{coordtr}). The Jacobian matrix shown in Proposition \ref{atlasintegr} shows that they are of the form
$$ \tilde{x}^i=\tilde{x}^i(x^j),
\,\tilde{y}^i=\frac{\partial\tilde{x}^i}{\partial x^j}y^j+\varphi^i(x),\,
\tilde{z}_i=\frac{\partial{x}^j}{\partial \tilde{x}^i}z_j+\psi_i(x).$$ These formulas show a locally affine structure on the symplectic leaves of $P$. The obstructions for an integrable big-tangent manifold to actually be the big tangent manifold $\mathfrak{T}M$ of a given manifold $M$ are given by the more general theory developed in \cite{VLagr}.}\end{rem}
\section{Horizontal bundles and Bott connections}
Like for the usual tangent bundle and like in foliation theory, the geometry of the manifold $\mathfrak{T}M$ may be enhanced by the addition of a {\it horizontal bundle} $\mathcal{H}$ such that
\begin{equation}\label{descbig} T\mathfrak{T}M=\mathcal{H}\oplus\mathcal{V}= \mathcal{H}\oplus\mathcal{V}_1\oplus\mathcal{V}_2.\end{equation}

The first equality (\ref{descbig}) produces a double grading of forms and multivectors; our convention is that {\it bidegree} or {\it type} $(p,q)$ means $\mathcal{H}$-degree $p$ and $\mathcal{V}$-degree $q$. The exterior differential	 has the decomposition
\begin{equation}\label{descd} d=d'_{(1,0)}+d''_{(0,1)}+\partial_{(2,-1)},\end{equation}
where $d''$ is the exterior differential along the leaves of $\mathcal{V}$ \cite{V71}. The second equality (\ref{descbig}), produces a double grading $(r,s)$ ($q=r+s$) of the $\mathcal{V}$-degree $q$, which corresponds to the terms $\mathcal{V}_1,\mathcal{V}_2$. This leads to a further decomposition of the terms of (\ref{descd}). For instance, we get $$ d''=d_{0,1,0}+d_{0,0,1}.$$

For a chosen $\mathcal{H}$, a vector $X\in T_xM$ has a {\it horizontal lift} $X^h$ defined by $X^h\in\mathcal{H}_{(x,y,z)}$,$p_*X^h=X$, and $T\mathfrak{T}M$ has local canonical bases
$$ X_i=\left(\frac{\partial}{\partial x^i}\right)^h= \frac{\partial}{\partial x^i}-t_i^j\frac{\partial}{\partial y^j}-\tau_{ij}\frac{\partial}{\partial z_j},\, \frac{\partial}{\partial y^i},\,\frac{\partial}{\partial z_i},$$
where $t,\tau$ are local functions of $(x,y,z)$.
The corresponding dual bases are
\begin{equation}\label{hcobase} dx^i,\,\theta^i=dy^i+t^i_jdx^j,\,\kappa_i=dz_i+\tau_{ij}dx^j.
\end{equation} A change of coordinates (\ref{coordtr}) implies the following transformation formulas
\begin{equation}\label{transh} \begin{array}{l}
\tilde{X}_i=\frac{\partial x^j}{\partial \tilde{x}^i}X_j,\;\tilde{\theta}^i=\frac{\partial \tilde{x}^i}{\partial{x}^j}\theta^j, \;\tilde{\kappa}_i=\frac{\partial {x}^j}{\partial{\tilde{x}}^i}\kappa_j,\vspace*{2mm}\\ \tilde{t}^i_j=\frac{\partial\tilde{x}^i}{\partial x^h} \frac{\partial x^k}{\partial \tilde{x}^j}t^h_k -\frac{\partial x^k}{\partial \tilde{x}^j}\frac{\partial^2 \tilde{x}^i}{\partial{x}^l\partial x^k}y^l,\; \tilde{\tau}_{ij}=\frac{\partial x^h}{\partial \tilde{x}^i} \frac{\partial x^k}{\partial \tilde{x}^j}\tau_{hk} -\frac{\partial^2 x^h}{\partial{\tilde{x}}^i\partial\tilde{x}^j}z_h.
\end{array}\end{equation}
\begin{prop}\label{extdehor} Any horizontal bundle on $TM$ has a canonical lift to a horizontal bundle on $\mathfrak{T}M$. Any horizontal bundle on the cotangent bundle $T^*M$ has a canonical lift to a horizontal bundle on $\mathfrak{T}M$.\end{prop}
\begin{proof} In both cases, the horizontal bundle of $\mathfrak{T}M$ is spanned by the extended lifts of the projectable, horizontal vector fields of $TM$, respectively, $T^*M$.
If the basis of the horizontal lift of $TM$ is $\partial/\partial x^i-t_i^j(x,y)(\partial/\partial y^j)$, then, formula (\ref{extend2}) gives the following horizontal basis of the canonical lift
$$\frac{\partial}{\partial x^i}-t_i^j\frac{\partial}{\partial y^j}+z_h\frac{\partial t_i^h}{\partial y^j}
\frac{\partial}{\partial z_j}.$$
Similarly, if the horizontal basis on $T^*M$ is
$\partial/\partial x^i-\tau_{ij}(x,z)(\partial/\partial z_j)$, then, formula (\ref{extend3}) gives the following horizontal basis of the canonical lift
$$\frac{\partial}{\partial x^i}+ z_h\frac{\partial\tau_{ih}}
{\partial z_j}\frac{\partial}{\partial y^j} -\tau_{ij}\frac{\partial}{\partial z_j}.$$

Alternatively, as indicated in Remark \ref{obseliftgen}, instead of using the extended lift we may use extensions from the pullback bundles to $\mathfrak{T}$. This implies that if coefficients $t_i^j$ satisfying (\ref{transh}) are known, the formula
\begin{equation}\label{taulift}\tau_{ij}=-z_h\frac{\partial t_i^h}{\partial y^j}\end{equation}
yields coefficients $\tau_{ij}$ that satisfy (\ref{transh}), hence, we get a horizontal bundle on $\mathfrak{T}M$. Similarly, if we have the coefficients $\tau_{ij}$ the formula
\begin{equation}\label{tlift}t_i^j=-z_h\frac{\partial\tau_{ih}}
{\partial z_j}\end{equation} completes the construction of a horizontal bundle.\end{proof}
\begin{rem}\label{nonlincovar} {\rm
In a different terminology (e.g., \cite{BM}), $\mathcal{H}$ is called a non-linear connection, $t,\tau$ are called the connection coefficients and there exists an equivalent covariant derivative $\nabla^{\mathcal{H}}_X:\Gamma(p^{-1}(\mathbf{T}M))\rightarrow \Gamma(p^{-1}(\mathbf{T}M))$ ($X\in TM$), which has the local expression
$$\nabla^{\mathcal{H}}_{\xi^k\frac{\partial}{\partial x^k}}(\nu^i\frac{\partial}{\partial x^i},\kappa_jdx^j)=(\xi^j(\frac{\partial\nu^i}{\partial x^j}+t^i_j)\frac{\partial}{\partial x^i},\xi^j(\frac{\partial\kappa_i}{\partial x^j}-\tau_{ij})dx^i).$$}\end{rem}
\begin{example}\label{exHGamma} {\rm
Let $\Gamma$ be an arbitrary linear connection on $M$ with local connection coefficients $\Gamma^i_{jk}$. Then, the tangent vectors of the paths defined in $\mathfrak{T}M$ by	 parallel translation of the vector $y$ and of the covector $z$ along paths through $x$ in $M$ span a complement $\mathcal{H}_{(x,y,z)}$ of $\mathcal{V}_{(x,y,z)}$. The horizontal distribution obtained in this way has the local bases
\begin{equation}\label{Gammahoriz} X_i=\frac{\partial}{\partial x^i}-y^k\Gamma^j_{ik} \frac{\partial}{\partial y^j} +z_k\Gamma^k_{ij} \frac{\partial}{\partial z_j}.\end{equation}
The coefficients of this connection are related by (\ref{taulift}) and (\ref{tlift}) simultaneously.}\end{example}
\begin{example}\label{exHL} {\rm We recall the following classical construction on $TM$ \cite{{Cra},{MR}}.
A vector field
\begin{equation}\label{vectso} \mathfrak{X}=y^i\frac{\partial}{\partial{x}^i}
+\eta^i\frac{\partial}{\partial{y}^i}\in\chi(TM)\end{equation}
is {\it of the second order} because it is locally equivalent with a system of second order, ordinary, differential equations on $M$. For any second order vector field $\mathcal{X}$, the $(1,1)$-tensor field $L_{\mathcal{X}}S_{TM}$ is an almost product structure ($(L_{\mathcal{X}}S_{TM})^2=Id$) with the $+1$-eigenbundle equal to the vertical bundle. Hence, the $-1$-eigenbundle is a horizontal bundle $\mathcal{H}_{\mathcal{X}}$.

Let $\mathcal{L}\in C^\infty(TM)$ be a function with a non degenerate Hessian matrix $(\partial^2\mathcal{L}/\partial y^i\partial y^j)$, called a {\it regular Lagrangian}. Then, the equation
$i(\Gamma_{\mathcal{L}})\theta
=-d\mathcal{E}$, where
$$\theta=d\vartheta,\;
\vartheta=d\mathcal{L}\circ S_{TM}=\frac{\partial\mathcal{L}}{\partial y^i}dx^i,\;
\mathcal{E}=y^i\frac{\partial\mathcal{L}}{\partial y^i} -\mathcal{L},$$	 defines a second order vector field $\Gamma_{\mathcal{L}}$, therefore, a corresponding horizontal bundle $\mathcal{H}_{\Gamma_{\mathcal{L}}}$ on $TM$. The extended lift of $\mathcal{H}_{\Gamma_{\mathcal{L}}}$ given by Proposition \ref{extdehor} is a horizontal bundle $\mathcal{H}_{\mathcal{L}}$ on $\mathfrak{T}M$.
Notice that the field $\Gamma_{\mathcal{L}}$ itself is not foliated, hence, it does not have an extended lift.}\end{example}

We can give a $\mathfrak{T}M$-version of the notion of a second order vector field of Example \ref{exHL} and say that the vector field (\ref{vect}) is of {\it the second order} if $S\mathfrak{X}=\mathcal{E}_1$, where $\mathcal{E}_1$ is defined in (\ref{EVF}), equivalently, $\mathfrak{X}$ has the local expression
\begin{equation}\label{vectq} \mathfrak{X}=y^i\frac{\partial}{\partial{x}^i}
+\eta^i\frac{\partial}{\partial{y}^i}
+\zeta_i\frac{\partial}{\partial{z}_i}.\end{equation}
\begin{prop}\label{vect2pebig} For any second order vector field $\mathfrak{X}$ on $\mathfrak{T}M$, the tensor field
$Q_{\mathfrak{X}}=L_{\mathfrak{X}}S\in End(T\mathfrak{T}M)$ satisfies the identity $Q_{\mathfrak{X}}^3-Q_{\mathfrak{X}}=0$. The $-1$-eigenbundle of $Q_{\mathfrak{X}}$ is a horizontal bundle $\mathcal{H}_{\mathfrak{X}}$ with the local bases
\begin{equation}\label{vectorhor}
X_i= \frac{\partial}{\partial x^i}+\frac{1}{2}\frac{\partial\eta^j}{\partial y^i}\frac{\partial}{\partial y^j}+\frac{\partial\zeta_j}{\partial y^i}\frac{\partial}{\partial z_j}.\end{equation}
\end{prop}
\begin{proof} Easy calculations show that the required identity holds on the natural tangent basis defined by $(x^i,y^i,z_i)$. In particular, $Q_{\mathfrak{X}}(\partial/\partial y^i)=\partial/\partial y^i$, $Q_{\mathfrak{X}}(\partial/\partial z_i)=0$ and $\mathcal{V}$ is the direct sum of the eigenspaces with eigenvalues $1,0$ of $Q_{\mathfrak{X}}$. Hence, if we denote the $-1$-eigenbundle by $\mathcal{H}_{\mathfrak{X}}$, the latter is horizontal. The local coordinate computation of $Q_{\mathfrak{X}}\mathfrak{X}'$ shows that $\mathfrak{X}'$ is a $-1$-eigenvector iff it is of the form
$$ \mathfrak{X}'=\xi^{'i}\frac{\partial}{\partial x^i}
+\frac{1}{2}\xi^{'j}\frac{\partial\eta^i}{\partial y^j} \frac{\partial}{\partial y^i}
+\xi^{'j}\frac{\partial\zeta_i}{\partial y^j} \frac{\partial}{\partial z_i},$$
therefore, the horizontal lifts of $\partial/\partial x^i$, are (\ref{vectorhor}).\end{proof}
\begin{rem}\label{obsobsHX1} {\rm For any choice of a horizontal bundle $\mathcal{H}$ on $\mathfrak{T}M$, the vector field $\mathfrak{X}=y^i\mathfrak{X}_i$, where $\mathfrak{X}_i$ is the horizontal lift of $\partial/\partial x^i$, is the unique, $\mathcal{H}$-horizontal, second order vector field.}\end{rem}
\begin{rem}\label{obsextendq} {\rm The transformation formulas (\ref{vectortr}) allow us to check that, if (\ref{vectso}) is a second order vector field (or, more generally, a cross section of $p_1^{-1}(T(TM))$, i.e., the coefficients $\eta^i$ also depend on the ``parameters" $z_j$), it has a canonical extension to a second order vector field on $\mathfrak{T}M$, which is given by
$$\mathfrak{X}=y^i\frac{\partial}{\partial{x}^i}
+\eta^i\frac{\partial}{\partial{y}^i}-\frac{1}{2} \frac{\partial\eta^h}{\partial y^i}z_h \frac{\partial}{\partial{z}_i}.$$}\end{rem}
\begin{rem}\label{obsafin} {\rm
On both $TM$ and $\mathfrak{T}M$, the set $\mathcal{Q}(TM)$, respectively $\mathcal{Q}(\mathfrak{T}M)$, of second order vector fields is closed under convex linear combinations and its geometric interpretation is that of the set of global cross sections of an affine fiber bundle modeled over the vertical vector bundle. In fact, this is an affine subbundle of the affine tangent bundle of the manifold $TM$, respectively $\mathfrak{T}M$ \cite{KN}. The points of the affine fibers are vectors of the form (\ref{vectso}), respectively (\ref{vectq}), where the first term is the origin and the two other terms define the vector from the origin to the point $\mathfrak{X}$. The coefficients $\eta^i,\zeta_i$ are fiber-wise affine coordinates.}\end{rem}
\begin{rem}\label{formaLi} {\rm
On $\mathfrak{T}M$, we can also define a special class of $1$-forms, namely, a $1$-form $\mathfrak{a}$ will be {\it Liouville-related} if $\mathfrak{a}\circ S$ is the Liouville form $\lambda$. Then locally $\mathfrak{a}$ looks as follows
$$\mathfrak{a}=\alpha_idx^i+z_idy^i+\gamma^idz_i.$$}\end{rem}

For the rest of this section we fix a horizontal bundle $\mathcal{H}$. Then, foliation theory offers the following important linear connections.
\begin{defin}\label{defBott} {\rm A {\it connection with no mixed torsion} or {\it Bott connection} is a linear connection on $\mathfrak{T}M$ that preserves the subbundles $\mathcal{H},\mathcal{V}$ and its torsion $T_\nabla$ satisfies the condition
$T_\nabla(\mathfrak{X},\mathfrak{Y})=0$, $\forall\mathfrak{X}\in\mathcal{H},\mathfrak{Y} \in\mathcal{V}$.}\end{defin}

The required torsion condition is equivalent to
\begin{equation}\label{BottV} \nabla_{\mathfrak{X}}\mathfrak{Y} =pr_{\mathcal{V}}[\mathfrak{X},\mathfrak{Y}],\,
\nabla_{\mathfrak{Y}}\mathfrak{X} =pr_{\mathcal{H}}[\mathfrak{Y},\mathfrak{X}]\end{equation}
and also equivalent to
\begin{equation}\label{BottV'} \nabla_{\mathfrak{X}}\mathfrak{Y} =[\mathfrak{X},\mathfrak{Y}],\,
\nabla_{\mathfrak{Y}}\mathfrak{X}=0, \end{equation}
where the horizontal field $\mathfrak{X}$ is projectable by $p$.

If $D$ is an arbitrary, linear connection on $\mathfrak{T}M$, the addition of the derivatives
\begin{equation}\label{BottDV}
\nabla_{\mathfrak{X}}\mathfrak{X}'
=pr_{\mathcal{H}}D_{\mathfrak{X}}\mathfrak{X}',\, \nabla_{\mathfrak{Y}}\mathfrak{Y}'
=pr_{\mathcal{V}}D_{\mathfrak{Y}}\mathfrak{Y}'\; (\mathfrak{X},\mathfrak{X}'\in\mathcal{H},\,
\mathfrak{Y},\mathfrak{Y}'\in\mathcal{V})
\end{equation}
to (\ref{BottV}) defines a Bott connection $\nabla^D=\nabla$. Following \cite{Bej0}, which traced back the history of this connection to a 1931 paper by G. Vr\u anceanu \cite{Vr}, we call $\nabla^D$ a {\it Vr\u anceanu-Bott connection}.
Furthermore, if $D$ has zero torsion, $\nabla^D$ has the torsion
\begin{equation}\label{torsBott} \begin{array}{r}
T_{\nabla^D}(\mathfrak{Z}_1,\mathfrak{Z}_2) =-pr_{\mathfrak{V}}[pr_{\mathcal{H}}\mathfrak{Z}_1, pr_{\mathcal{H}}\mathfrak{Z}_2]\vspace*{2mm}\\ = -R_{\mathcal{H}}(\mathfrak{Z}_1,\mathfrak{Z}_2),\;\; \mathfrak{Z}_1,\mathfrak{Z}_2\in\chi(\mathfrak{T}M).\end{array}
\end{equation} $R_{\mathcal{H}}$ is called the Ehressmann curvature of the non-linear connection $\mathcal{H}$.
\begin{defin}\label{def2Bott} {\rm A {\it connection with no multi-mixed torsion} on $\mathfrak{T}M$ is a Bott connection $\nabla$ that also preserves the subbundles $\mathcal{V}_1,\mathcal{V}_2$ and satisfies the condition
$ T_\nabla(\mathfrak{X}_a,\mathfrak{Y}_{a'})=0$, where
$a,a'\in\{1,2\}$, $a'\equiv a+1\,({\rm mod.}\,2)$ and
$\mathfrak{X}_a,\mathfrak{Y}_a\in\mathcal{V}_a$.}\end{defin}

The new torsion condition of Definition \ref{def2Bott} is equivalent to
\begin{equation}\label{Bott}\nabla_{\mathfrak{X}_a}\mathfrak{Y}_{a'} =pr_{\mathcal{V}_{a'}}[\mathfrak{X}_a,\mathfrak{Y}_{a'}].\end{equation}
Notice that for a $\mathcal{V}$-projectable vector field $\mathfrak{X}$, the first equation (\ref{BottV'}) becomes $\nabla_{\mathfrak{X}}\mathfrak{Y}_a =pr_{\mathcal{V}_a}[\mathfrak{X},\mathfrak{Y}_a]$.
\begin{prop}\label{canonicBott} On $\mathfrak{T}M$, there exists a {\it canonical connection} $\nabla$ with no multi-mixed torsion, which depends only on the horizontal bundle $\mathcal{H}$ and such that, along the leaves of the foliations $\mathcal{V}_1,\mathcal{V}_2,\mathcal{V}$, this is the flat connection of the affine structure of the leaves.\end{prop}
\begin{proof} Consider the isomorphism $\mathcal{S}=S|_{\mathcal{H}}:\mathcal{H}\rightarrow\mathcal{V}_1$ and the transposed isomorphism $\hspace{1pt}^t\hspace{-1pt} \mathcal{S}:\mathcal{V}_2\rightarrow\mathcal{H}^*$
(remember that $\mathcal{V}_2\approx\mathcal{V}_1^*$). Then, add to (\ref{BottV}) and (\ref{Bott}) the following covariant derivatives
\begin{equation}\label{canBH} \begin{array}{c} \nabla_{\mathfrak{X}}\mathfrak{X}'=\mathcal{S}^{-1}pr_{\mathcal{V}_1}
[\mathfrak{X},\mathcal{S}\mathfrak{X}'],\, \nabla_{\mathfrak{Y}_1}\mathfrak{Y}'_1=\mathcal{S}pr_{\mathcal{H}}
[\mathfrak{Y}_1,\mathcal{S}^{-1}\mathfrak{Y}'_1],\vspace*{2mm}\\ \nabla_{\mathfrak{Y}_2}\mathfrak{Y}'_2= (\hspace{1pt}^t \hspace{-1pt}\mathcal{S})^{-1} pr_{\mathcal{H}^*} L_{\mathfrak{Y}_2}(\hspace{1pt}^t\hspace{-1pt}\mathcal{S}
\mathfrak{Y}'_2),\end{array}\end{equation} $\forall\mathfrak{X},\mathfrak{X}'\in
\Gamma\mathcal{H},\mathfrak{Y}_a,\mathfrak{Y}'_a\in
\Gamma\mathcal{V}_a$ and $ pr_{\mathcal{H}^*}$ is defined by means of (\ref{descbig}). The result is a linear connection on $\mathfrak{T}M$, which depends only on $\mathcal{H}$. The affine structure of the leaves of $\mathcal{V}_1,\mathcal{V}_2,\mathcal{V}$ is shown by the coordinate transformations (\ref{coordtr}) and it is easy to see that the parallel vector fields of this structure are $\mathfrak{Y}'_1\in\mathcal{V}_1$ such that $\mathcal{S}^{-1}\mathfrak{Y}'_1$ is projectable, $\mathfrak{Y}'_2\in\mathcal{V}_2$ such that $\hspace{-1pt}\mathcal{S}\mathfrak{Y}'_2$ is projectable and sums of such vector fields, respectively. Then, (\ref{canBH}) implies $\nabla_{\mathfrak{Y}_1}\mathfrak{Y}'_1=0, \nabla_{\mathfrak{Y}_2}\mathfrak{Y}'_2=0$ and (\ref{Bott})	 implies $\nabla_{\mathfrak{X}_a}\mathfrak{Y}'_{a'}=0$ (the last conclusion follows by using a parallel extension of the point-value of $\mathfrak{X}_a$, which leads to vanishing brackets in the right hand side of (\ref{Bott})).
\end{proof}

Any linear connection $D$ on $\mathfrak{T}M$ has a well defined deformation into a	connection $\bar{\nabla}^D$ with no multi-mixed torsion, which is obtained by adding to (\ref{BottV}) and (\ref{Bott}) the first equation (\ref{BottDV}) and the derivatives
\begin{equation}\label{DBott}
\bar{\nabla}^D_{\mathfrak{X}_a}\mathfrak{Y}_{a}
=pr_{\mathcal{V}_{a}}D_{\mathfrak{X}_a}\mathfrak{Y}_{a}.
\end{equation}
Generally, $\bar{\nabla}^D\neq\nabla^D$.
If $D$ has no torsion, the torsion of $\bar{\nabla}^D$ is still given by (\ref{torsBott}).

We also recall the following notion from foliation theory
\begin{defin}\label{conexpr} {\rm A Bott connection $\nabla$ is said to be {\it projectable} if for any two projectable, horizontal vector fields $\mathfrak{X},\mathfrak{X}'$, $\nabla_{\mathfrak{X}}\mathfrak{X}'$ is projectable.}\end{defin}
\begin{example}\label{excpr} {\rm If $\mathcal{H}$ is the horizontal space defined in Example \ref{exHGamma}, then, the corresponding canonical connection given by Proposition \ref{canonicBott} is projectable. The result follows by applying the first formula (\ref{canBH}) to the vector fields
(\ref{Gammahoriz}).}\end{example}

Concerning curvature, we can extend results from foliation theory and we have
\begin{prop}\label{propgencurb} The curvature $R_\nabla$ of a connection $\nabla$ without mixed torsion has the following properties
\begin{equation}\label{curb0} \begin{array}{l}
R_{\nabla}(\mathfrak{Y},\mathfrak{Y}')\mathfrak{X}=0, \hspace*{2mm} R_{\nabla}(\mathfrak{Y},\mathfrak{X})\mathfrak{X}'=
pr_{\mathcal{H}}[\mathfrak{Y},\nabla_{\mathfrak{X}}\mathfrak{X}'] \vspace*{2mm}\\
R_{\nabla}(\mathfrak{X},\mathfrak{X}')\mathfrak{Y}=
T_\nabla(\mathfrak{Y},R_{\mathcal{H}}(\mathfrak{X},\mathfrak{X}'))
-\nabla_{\mathfrak{Y}}(R_{\mathcal{H}}(\mathfrak{X},\mathfrak{X}')),
\end{array}\end{equation}
where $\mathfrak{X},\mathfrak{X}'\in\mathcal{H},\, \mathfrak{Y},\mathfrak{Y}'\in\mathcal{V}$, $R_{\mathcal{H}}$ is the Ehressmann curvature of $\mathcal{H}$ and the right hand sides are computed by using projectable fields $\mathfrak{X}_1,\mathfrak{X}_2$.
If $\nabla$ has no multi-mixed torsion one also has
\begin{equation}\label{curb01} R_{\nabla}(\mathfrak{Y}_a,\mathfrak{Y}'_a)\mathfrak{Y}_{a'}=0, \end{equation}
where the indices are like in (\ref{Bott}) and $\mathfrak{Y}_a\in\mathcal{V}_a$.
\end{prop}
\begin{proof} Since only the point-wise values of the arguments count, we get
(\ref{curb0}) from (\ref{BottV'}) by assuming that $\mathfrak{X}$ is $\mathcal{V}$-projectable.
Similarly, we get (\ref{curb01}) from (\ref{Bott}) by assuming that $\mathfrak{Y}_{a'}$ is $\mathcal{V}_a$-projectable.
\end{proof}
\begin{corol}\label{corolcurbpr} For a projectable connection one has $R_\nabla(\mathfrak{Y},\mathfrak{X})\mathfrak{X}'=0$.
\end{corol} \begin{proof} The result follows from the second formula (\ref{curb0}).\end{proof}
\begin{prop}\label{propcurbVB} The curvature of the Vr\u anceanu-Bott connection $\nabla^D$, where $D$ is a torsionless linear connection on $\mathfrak{T}M$, also has the following properties
\begin{equation}\label{curb10} \begin{array}{l}
R_{\nabla^D}(\mathfrak{Y},\mathfrak{X})\mathfrak{X}'=
R_{\nabla}(\mathfrak{Y},\mathfrak{X}')\mathfrak{X},\vspace*{2mm}\\
R_{\nabla^D}(\mathfrak{X},\mathfrak{X}')\mathfrak{Y}= \nabla^D_{\mathfrak{Y}}(T_{\nabla^D}(\mathfrak{X},\mathfrak{X}')),	 \vspace*{2mm}\\
\sum_{Cycl(\mathfrak{X},\mathfrak{X}',\mathfrak{X}'')}
R_{\nabla^D}(\mathfrak{X},\mathfrak{X}')\mathfrak{X}''=0,\vspace*{2mm}\\ \sum_{Cycl(\mathfrak{Y},\mathfrak{Y}',\mathfrak{Y}'')}
R_{\nabla^D}(\mathfrak{Y},\mathfrak{Y}')\mathfrak{Y}''=0,
\vspace*{2mm}\\ R_{\nabla^D}(\mathfrak{X},\mathfrak{Y}) \mathfrak{Y}'=R_{\nabla^D}(\mathfrak{X},\mathfrak{Y}') \mathfrak{Y},
\end{array}\end{equation}
where all $\mathfrak{X}\in\mathcal{H}$ (and projectable when needed) and all
$\mathfrak{Y}\in\mathcal{V}$. The same properties hold for the connection with no multi-mixed torsion $\bar{\nabla}^D$ defined by the torsionless connection $D$. In particular, for $\bar{\nabla}^D$ we have $R_{\nabla^D}(\mathfrak{Y}_a,\mathfrak{Y}'_a)\mathfrak{Y}=0$,
$\forall\mathfrak{Y}_a\in\mathcal{V}_a,\mathfrak{Y}\in\mathcal{V}$, $a=1,2$.
\end{prop}
\begin{proof} The first and second formulas (\ref{curb10}) are a consequence of (\ref{curb0}) and of the expression (\ref{torsBott}) of the torsion. The three other formulas follow from (\ref{torsBott}) and the Bianchi identity
(\cite{KN}, vol. I)
$$ \sum_{Cycl(\mathfrak{X},\mathfrak{X}'\mathfrak{X}'')}
[R_{\nabla^D}(\mathfrak{X},\mathfrak{X}')\mathfrak{X}'' -
T(T(\mathfrak{X},\mathfrak{X}'),\mathfrak{X}'')-
(\nabla_{\mathfrak{X}}T)(\mathfrak{X}',\mathfrak{X}'')]=0.$$
We give some details for the case where all the arguments are horizontal, while assuming that they also are projectable.
Then, the first torsion term of the Bianchi formula vanishes since its first argument is vertical and the second is horizontal. If we use projectable vector fields $\mathfrak{X},\mathfrak{X}',\mathfrak{X}''$, the second torsion term of the Bianchi formula must be vertical and, since the curvature term is horizontal, we are done. The same properties hold for the connection $\bar{\nabla}^D$ because it has the same torsion like $\nabla^D$. The last assertion of the proposition follows by using the fourth relation (\ref{curb10}) for $\mathfrak{Y},\mathfrak{Y}'\in\mathcal{V}_a,\mathfrak{Y}'' \in\mathcal{V}_{a'}$ and by decomposing $\mathfrak{Y}$ into its $\mathcal{V}_a$-components.\end{proof}
\section{Metrics and double fields}
The general, local expression of a pseudo-Riemannian metric on the big-tangent manifold $\mathfrak{T}M$ is
$$ \begin{array}{r}\mathfrak{g}=
\hspace{1pt}^1\hspace{-1pt}g_{ij}dx^i\odot dx^j + \hspace{1pt}^2\hspace{-1pt}g_{ij}dx^i\odot dy^j+
\hspace{1pt}^3\hspace{-1pt}g^{j}_{i}dx^i\odot dz_j\vspace*{2mm}\\ +\hspace{1pt}^4\hspace{-1pt}g_{ij}dy^i\odot dy^j+
\hspace{1pt}^5\hspace{-1pt}g^{j}_{i}dy^i\odot dz_j+\hspace{1pt}^6\hspace{-1pt}g^{ij}dz_i\odot dz_j,\end{array}$$
where the coefficients $\hspace{1pt}^1\hspace{-1pt}g, \hspace{1pt}^4\hspace{-1pt}g, \hspace{1pt}^6\hspace{-1pt}g$ are symmetric.

Besides, we will always assume that the metric has a non degenerate restriction to the vertical bundle $\mathcal{V}$, which implies that the orthogonal bundle $\mathcal{H}_{\mathfrak{g}}\perp_{\mathfrak{g}}\mathcal{V}$ is a horizontal bundle and $\mathfrak{g}|_{\mathcal{H}}$ is non degenerate too. Hereafter, always, in the presence of a metric, we will use the horizontal bundle $\mathcal{H}=\mathcal{H}_{\mathfrak{g}}$.
\begin{example}\label{exSasaki} {\rm The formula
$$ \mathfrak{s}=g_{ij}dx^i\otimes dx^j+ g_{ij}Dy^i\otimes Dy^j+g^{ij}Dz_i\otimes Dz_j,
$$
where $g$ is a pseudo-Riemannian metric on $M$, $D$ is the Levi-Civita connection of $g$ and $Dy^j,Dz_i$ are the usual expressions of the covariant differential of a contravariant and covariant vector field, respectively, defines a metric on $\mathfrak{T}M$, which is non degenerate on $\mathcal{V}$ and will be called the {\it Sasaki metric}.}\end{example}
\begin{example}\label{exSasakitype} {\rm Let $g$ be a pseudo-Riemannian metric on $M$, $\mathcal{H}$ an arbitrary horizontal bundle and $(dx^i,\theta^i,\kappa_i)$ the corresponding cobasis (\ref{hcobase}). Then, the formula
$$\mathfrak{g}= g_{ij}dx^i\otimes dx^j+ g_{ij}\theta^i\otimes\theta^j +g^{ij}\kappa_i\otimes\kappa_j$$ defines a metric on $\mathfrak{T}M$, which is non degenerate on $\mathcal{V}$ and will be called a {\it Sasaki-type metric}. In particular, if
$g_{ij}=\partial^2\mathcal{L}/\partial y^i\partial y^j$, where $\mathcal{L}$ is a regular Lagrangian, we get a metric $\mathfrak{g}_{\mathcal{L}}$.}\end{example}

We recall the following result of foliation theory \cite{{V71},{V73}}:
\begin{prop}\label{prcancon}
Let $D$ be the Levi-Civita connection of $\mathfrak{g}$. The corresponding Vr\u anceanu-Bott connection $\nabla^D$ defined by (\ref{BottDV}) is the unique connection such that: i) the subbundles $\mathcal{H}$ and $\mathcal{V}$ are preserved ii) parallel translations along curves that are tangent to either $\mathcal{H}$ or $\mathcal{V}$ preserve the restriction of $\mathfrak{g}$ to $\mathcal{H},\mathcal{V}$, respectively, iii) the restrictions of the torsion to $\mathcal{H}$ and $\mathcal{V}$ take values in $\mathcal{V}$ and $\mathcal{H}$, respectively.\end{prop}
\begin{proof} The meaning of the properties i), ii), iii) is
$$\nabla^D_{\mathfrak{Z}}\mathfrak{X}\in\mathfrak{H},\;
\nabla^D_{\mathfrak{Z}}\mathfrak{Y}\in\mathfrak{V},\;{\rm for}\;\mathfrak{Z}\in\chi(\mathfrak{T}M),\, \mathfrak{X}\in\mathcal{H},\,
\mathfrak{Y}\in\mathcal{V},$$
$$\nabla^D_{\mathfrak{X}}\mathfrak{g}(\mathfrak{Y},\mathfrak{Z})=0,
\;{\rm if}\;\mathfrak{X},\mathfrak{Y},\mathfrak{Z}\in\mathcal{H}\;{\rm or}\;\mathfrak{X},\mathfrak{Y},\mathfrak{Z}\in\mathcal{V},$$
$$pr_{\mathcal{H}}T_{\nabla^D}(\mathfrak{X},\mathfrak{Y})=0\;{\rm if}\;\mathfrak{X},\mathfrak{Y}\in\mathcal{H},\;\;
pr_{\mathcal{V}}T_{\nabla^D}(\mathfrak{X},\mathfrak{Y})=0\;{\rm if}\;\mathfrak{X},\mathfrak{Y}\in\mathcal{V}.$$	 All these equalities follow from the definition of $\nabla^D$ and of the Levi-Civita connection $D$ .
Conversely, if the previous equalities hold, the trick that gives the global expression of the Levi-Civita connection (e.g., \cite{KN}, vol. I) yields $$\mathfrak{g}(\nabla^D_{\mathfrak{X}}\mathfrak{Y},\mathfrak{Z})=
\mathfrak{g}(pr_{\mathcal{H}}D_{\mathfrak{X}}\mathfrak{Y},\mathfrak{Z}),
\,\mathfrak{g}(\nabla^D_{\mathfrak{X}}\mathfrak{Y},\mathfrak{Z})=
\mathfrak{g}(pr_{\mathcal{V}}D_{\mathfrak{X}}\mathfrak{Y},\mathfrak{Z}),$$
where all the arguments belong to $\mathcal{H}$ in the first case and to $\mathcal{V}$ in the second case, whence, the required result.\end{proof}

Notice that the restriction of $\nabla^D$ to the leaves of $\mathcal{V}$ is the Levi-Civita connection of the leaves. The connection $\nabla^D$ above will be called the {\it canonical connection} of $\mathfrak{g}$ on $\mathfrak{T}M$.

The canonical connection does not preserve the subbundles $\mathcal{V}_1,\mathcal{V}_2$. Instead, the latter are preserved by the corresponding connection (\ref{DBott}), which may also be written as
\begin{equation}\label{nablabarD}
\bar{\nabla}^D_{\mathfrak{Z}}\mathfrak{U}= \nabla^D_{\mathfrak{Z}}\mathfrak{U},\,
\bar{\nabla}^D_{\mathfrak{X}_a}\mathfrak{Y}_a=pr_{\mathcal{V}_a} \nabla^D_{\mathfrak{X}_a}\mathfrak{Y}_a,\, \bar{\nabla}^D_{\mathfrak{X}_a}\mathfrak{Y}_{a'}=pr_{\mathcal{V}_{a'}}
[\mathfrak{X}_a,\mathfrak{Y}_{a'}],\end{equation}
where at least one of the vector fields $\mathfrak{Z},\mathfrak{U}$ is horizontal and $a,a'\in\{1,2\}$, $a'\equiv a+1 ({\rm mod.}\,2)$.

The torsion and curvature of the canonical connection $\nabla^D$ satisfy the formulas (\ref{torsBott}), (\ref{curb0}) and (\ref{curb10}) and these imply corresponding properties for the covariant curvature tensor
$$R_{\nabla^D}(\mathfrak{Z}_1, \mathfrak{Z}_2,\mathfrak{Z}_3,\mathfrak{Z}_4) =\mathfrak{g}(R_{\nabla^D}(\mathfrak{Z}_3,\mathfrak{Z}_4)\mathfrak{Z}_2, \mathfrak{Z}_1),$$ particularly,
$$\begin{array}{c}
R_{\nabla^D}(\mathfrak{Z}_1,\mathfrak{Z}_2,
\mathfrak{Z}_3,\mathfrak{Z}_4) =-R_{\nabla^D}(\mathfrak{Z}_1,\mathfrak{Z}_2, \mathfrak{Z}_4,\mathfrak{Z}_3),\vspace*{2mm}\\
\sum_{Cycl(1,2,3)} R_{\nabla^D}(\mathfrak{Z},\mathfrak{X}_1, \mathfrak{X}_2,\mathfrak{X}_3)=0,\end{array}$$
where $\mathfrak{X}$ are horizontal vectors, $\mathfrak{Y}$ are vertical vectors and  $\mathfrak{Z}$ are arbitrary vectors.)

For other identities one needs a horizontal tensor $C$ called the {\it Cartan tensor} (a name suggested by Finsler geometry). This tensor is defined by \cite{VL}
$$ C(\mathfrak{X}_1,\mathfrak{X}_2,\mathfrak{X}_3) =(L_{S\mathfrak{X}_1}g)(\mathfrak{X}_2,\mathfrak{X}_3)
=(\nabla^D_{S\mathfrak{X}_1}g)(\mathfrak{X}_2,\mathfrak{X}_3),$$
where $g=\mathfrak{g}|_{\mathcal{H}}$ is a horizontal metric extended by $0$ on vertical arguments. The local components of $C$ are
$$ C_{ijk}=C(X_i,X_j,X_k)=\frac{\partial g_{jk}}{\partial y^i},$$ whence, we deduce that
$C=0$  iff $g$ is a projectable metric and $C$ is totally symmetric iff, locally (but, possibly, not globally), $g$ is the Hessian of a function ``in the direction of	 $\mathcal{V}_1$".

The following curvature identities may be proven like in \cite{VL}:
$$\begin{array}{l}
R_{\nabla^D}(\mathfrak{X}_1,\mathfrak{X}_2,\mathfrak{X}_3,\mathfrak{X}_4) +R_{\nabla^D}(\mathfrak{X}_2,\mathfrak{X}_1,\mathfrak{X}_3,\mathfrak{X}_4) \vspace*{2mm}\\ =-C(\mathcal{S}^{-1}T_{\nabla^D}
(\mathfrak{X}_3,\mathfrak{X}_4),\mathfrak{X}_1,\mathfrak{X}_2),
\vspace*{2mm}\\
R_{\nabla^D}(\mathfrak{X}_1,\mathfrak{X}_2,\mathfrak{X}_3,\mathfrak{X}_4) -R_{\nabla^D}(\mathfrak{X}_3,\mathfrak{X}_4,\mathfrak{X}_1,\mathfrak{X}_2) \vspace*{2mm}\\ =\frac{1}{2}[C(\mathcal{S}^{-1}T_{\nabla^D}
(\mathfrak{X}_1,\mathfrak{X}_2),\mathfrak{X}_3,\mathfrak{X}_4)
-C(\mathcal{S}^{-1}T_{\nabla^D}
(\mathfrak{X}_3,\mathfrak{X}_4),\mathfrak{X}_1,\mathfrak{X}_2)].
\end{array}$$
\begin{corol}\label{corolcurb} If either the horizontal component of the metric $\mathfrak{g}$ is projectable or the horizontal bundle is integrable, the restriction of the curvature of $\nabla^D$ to horizontal arguments satisfies the Riemannian curvature identities.\end{corol}
\begin{proof} The first case is characterized by $C=0$ and the second case by $T_{\nabla^D}=0$.
\end{proof}
\begin{defin}\label{vm} {\rm
A non degenerate metric $\mathfrak{g}_{\mathcal{V}}$ on the vertical bundle $\mathcal{V}$ is called a {\it vertical metric} on $\mathfrak{T}M$. A vertical metric that has a non degenerate restriction to $\mathcal{V}_2$ is called {\it strongly non degenerate}\footnote{The choice of $\mathcal{V}_2$ rather than $\mathcal{V}_1$ in this definition is geometrically more convenient, but it is not essential. In any case, there is no reason to ask for both restrictions to be non degenerate.}.}\end{defin}

If non degenerate (which was assumed), the restriction to $\mathcal{V}$ of a pseudo-Riemannian metric of $\mathfrak{T}M$ is a vertical metric. Conversely, any choice of a horizontal bundle $\mathcal{H}$ with a metric $\mathfrak{g}_{\mathcal{H}}$ allows us to extend the vertical metric $\mathfrak{g}_{\mathcal{V}}$ to the metric $\mathfrak{g}= \mathfrak{g}_{\mathcal{H}}+\mathfrak{g}_{\mathcal{V}}$.

Hereafter, all the vertical metrics are assumed strongly non degenerate.

In this section we will discuss vertical metrics that are compatible with the para-Hermitian metric of $\mathcal{V}$, where the compatibility condition is like in \cite{VT} and comes from generalized Riemannian metrics \cite{Galt}. The results too are based on these sources.

It is convenient to use $\mathcal{V}\approx\mathcal{V}_1\oplus\mathcal{V}_1^*$ and, correspondingly, write vertical vectors as pairs $(\mathfrak{Y},\mathfrak{a})$, where $\mathfrak{Y}$ is the $\mathcal{V}_1$-component and $\mathfrak{a}\in\mathcal{V}_1^*$ is the image of the $\mathcal{V}_2$-component.
Then, the restrictions $h=\mathfrak{g}_{\mathcal{V}}|_{\mathcal{V}_1\times\mathcal{V}_1},
k=\mathfrak{g}_{\mathcal{V}}|_{\mathcal{V}_2\times\mathcal{V}_2},
l=\mathfrak{g}_{\mathcal{V}}|_{\mathcal{V}_1\times\mathcal{V}_2}$ appear as tensors of type $(0,2),(1,1),(2,0)$ of the bundle $\mathcal{V}_1$.
\begin{rem}\label{obsLville} {\rm The strong non degeneracy condition means that $k$ is non degenerate and there exists a Legendre-type involution on $\mathfrak{T}M$ given by
$$(x,y,z)\mapsto(x,\sharp_kz,\flat_ky).$$}\end{rem}
\begin{example}\label{exHess} {\rm Let $\mathcal{K}$ be a function in $C^\infty(\mathfrak{T}M)$. $Hess\,\mathcal{K}(\mathfrak{Z},\mathfrak{Z}')= \mathfrak{Z}'\mathfrak{Z}\mathcal{K}$, $\mathfrak{Z},\mathfrak{Z}'\in\mathcal{V}$ restricted to
$\mathcal{V}_1\times\mathcal{V}_1,\mathcal{V}_2\times\mathcal{V}_2, \mathcal{V}_1\times\mathcal{V}_2$ yields tensor fields $h,l,k$. Thus, if the Hessian is non degenerate, it defines a vertical metric, which is strongly non degenerate if $k$ is non degenerate. Moreover, if we chose a horizontal bundle $\mathcal{H}$ and transfer $k$ to $\mathcal{H}$ by $(S|_{\mathcal{H}})^{-1}$ we obtain a metric of $\mathfrak{T}M$ associated to the function $\mathcal{K}$.}\end{example}

Consider the endomorphism $\phi\in End(\mathcal{V})$ defined by \begin{equation}\label{phi0}
2g(\phi\mathfrak{Z},\mathfrak{Z}')= \mathfrak{g}_{\mathcal{V}}(\mathfrak{Z},\mathfrak{Z}'),\;
\mathfrak{Z}\leftrightarrow(\mathfrak{Y},\mathfrak{a}) ,\mathfrak{Z}'
\leftrightarrow(\mathfrak{Y}',\mathfrak{a}')\in\mathcal{V},\end{equation}
where $g$ is the para-Hermitian metric of $\mathcal{V}$. The symmetry of $\mathfrak{g}_{\mathcal{V}}$ implies the property
\begin{equation}\label{symphi} g(\phi\mathfrak{Z},\mathfrak{Z}')=
g(\mathfrak{Z},\phi\mathfrak{Z}').
\end{equation}

From  (\ref{phi0}) we get the following matrix representation of $\phi$
\begin{equation}\label{phi2}
\phi\left(\begin{array}{c}\mathfrak{Y}\vspace{2mm}\\ \mathfrak{a} \end{array}
\right)=\left(\begin{array}{cc} l&\sharp_k\vspace{2mm}\\
\flat_h&^t\hspace{-1pt}l\end{array}\right) \left(
\begin{array}{c}\mathfrak{Y}\vspace{2mm}\\ \mathfrak{a} \end{array}\right).
\end{equation}
\begin{defin}\label{defcompg} {\rm\cite{Galt}
The vertical metric $\mathfrak{g}_{\mathcal{V}}$ is {\it $g$-compatible} if
$\phi^2=Id.$}\end{defin}

If $\phi^2=Id$ and, with the same notation, (\ref{phi0}) and (\ref{symphi}) imply
\begin{equation}\label{phi01} 2g(\mathfrak{Z},\mathfrak{Z}')= \mathfrak{g}_{\mathcal{V}}(\phi\mathfrak{Z},\mathfrak{Z}'),
\;\mathfrak{g}_{\mathcal{V}}(\phi\mathfrak{Z},\mathfrak{Z}')=
\mathfrak{g}_{\mathcal{V}}(\mathfrak{Z},\phi\mathfrak{Z}').
\end{equation}

In terms of $h,l,k$ the condition $\phi^2=Id$ means \begin{equation}\label{condcompat}
l^2+\sharp_k\circ\flat_h=Id,\, l\circ\sharp_k+\sharp_k\circ\hspace{1pt}^t\hspace{-1pt}l=0,\,
\flat_h\circ l+\hspace{1pt}^t\hspace{-1pt}l\circ\flat_h=0.
\end{equation}
\begin{prop}\label{Gcupsi} The strongly non degenerate,vertical $g$-compatible metrics $\mathfrak{g}_{\mathcal{V}}$ are in a bijective correspondence with pairs $(\sigma,\psi)$ where $\sigma$ is a non degenerate metric on $\mathcal{V}_1$ and $\psi\in\Gamma\wedge^2\mathcal{V}_1^*$. \end{prop}
\begin{proof} We proceed like in the case of generalized Riemannian metrics \cite{Galt}. By (\ref{phi0}), instead of looking at $\mathfrak{g}_{\mathcal{V}}$, we may look at the corresponding $\phi\in End(\mathcal{V})$, then, define the corresponding pair by\footnote{In fact, $\sigma$ is the covariant version of $k$ but, we prefer to use the new symbol to avoid any possible confusion.}
$$ \sigma=k^{-1},\;\flat_\psi=-\flat_\sigma\circ l$$
(the skew symmetry of $\psi$ follows from (\ref{condcompat})).
Conversely, given a pair $(\sigma,\psi)$, if we take
\begin{equation}\label{eqGcupsi} k=\sigma^{-1},\,l=-\sharp_\sigma\flat_\psi,\, \flat_h=\flat_\sigma\circ
(Id-(\sharp_\sigma\flat_\psi)^2),\end{equation}
we get the endomorphism $\phi$ and the required metric $\mathfrak{g}_{\mathcal{V}}$.
\end{proof}
\begin{rem}\label{obssigmapsi} {\rm	 A vertical metric $\mathfrak{g}_{\mathcal{V}}$ may be interpreted as a metric on the vector bundle $p^{-1}(\mathbf{T}M)\approx\mathcal{V}$ and
the tensors $\sigma,\psi$ may be interpreted as tensors of the vector bundle $p^{-1}(TM)\approx\mathcal{V}_1$.}\end{rem}
\begin{prop}\label{strUpm} Let $\mathcal{U}_\pm\subset\mathcal{V}$ be the $(\pm1)$-eigenbundles of the product structure $\phi$ that defines a compatible metric $\mathfrak{g}_{\mathcal{V}}$. Then, $\mathcal{U}_+\perp_{\mathfrak{g}_{\mathcal{V}}}\mathcal{U}_-$, the
projections $pr_{\mathcal{V}_1}|_{\mathcal{U}_\pm}: {\mathcal{U}_\pm}\rightarrow\mathcal{V}_1$ defined by the decomposition $\mathcal{V}=\mathcal{V}_1\oplus\mathcal{V}_2$ are isomorphisms, and the corresponding pullbacks of the metric $\sigma$ are equal to $(1/2)\mathfrak{g}_{\mathcal{V}}|_{\mathcal{U}_\pm}$.\end{prop}
\begin{proof} We follow \cite{VT}. The first conclusion is a consequence of (\ref{phi01}).
For the second, we first show that $\mathcal{U}_\pm\cap\mathcal{V}_2=0$. Indeed, if $\mathfrak{Y}\in\mathcal{U}_\pm\cap\mathcal{V}_2$, then,
$$\mathfrak{g}_{\mathcal{V}}(\mathfrak{Y},\mathfrak{Z})=
2g(\phi\mathfrak{Y},\mathfrak{Z})=
\pm 2g(\mathfrak{Y},\mathfrak{Z})=0, \;\forall \mathfrak{Z}\in\mathcal{V}_2,$$
because $\mathcal{V}_2$ is $g$-isotropic. Since $\mathfrak{g}_{\mathcal{V}}$ is strongly non degenerate, we deduce $\mathfrak{Y}=0$.
Together with the orthogonality between $\mathcal{U}_\pm$, this result implies $rank\,\mathcal{U}_\pm= m$, therefore,
$ \mathcal{V}_2\oplus \mathcal{U_\pm}=\mathcal{V}$ and
$pr_{\mathcal{V}_1}|_{\mathcal{U}_\pm}: {\mathcal{U}_\pm}\rightarrow\mathcal{V}_1$ are isomorphisms with inverses, say, $\iota_\pm$. The expression of $\iota_\pm$ was established in \cite{{Galt},{VT}} and is given by
$$ \iota_\pm\mathfrak{Y}=(\mathfrak{Y},\flat_k(l\mp Id)\mathfrak{Y}) =(\mathfrak{Y},(\flat_\psi\pm\flat_\sigma)\mathfrak{Y})$$
($\mathfrak{Y}\in\mathcal{V}_1$ and
the pairs correspond to vertical vectors). The last assertion of the proposition is a consequence of this expression.
\end{proof}
\begin{rem}\label{obsauxgUpm} {\rm Similarly, one gets
$$\mathcal{U}_+\perp_g\mathcal{U}_-,\;g|_{\mathcal{U}_\pm} =\pm(pr_{\mathcal{V}_1})^*\sigma.$$ }\end{rem}

On para-Hermitian manifolds, compatible, generalized, pseudo-Riemannian metrics may be seen as double fields, because of their similarity with the double fields of string theory (see \cite{VT} for details and references to string theory literature). This suggests the following definition, which, perhaps, will interest physicists.
\begin{defin}\label{defdfield} {\rm A {\it double field over a manifold} $M$ is a pair $(\mathcal{H},\mathfrak{g}_{\mathcal{V}})$ where $\mathcal{H}$ is a horizontal bundle on $\mathfrak{T}M$ and $\mathfrak{g}_{\mathcal{V}}$ is a compatible metric on $\mathcal{V}$.}\end{defin}

By Proposition \ref{Gcupsi} double fields over $M$ are in a bijective correspondence with triples $(\mathcal{H}, \sigma,\psi)$ where $\sigma$ is a non degenerate metric on $p^{-1}(TM)\approx\mathcal{V}_1$ and $\psi$ is a $2$-form on the same bundle. $\sigma$ and $\psi$ are called the {\it components of the field}.
\begin{example}\label{exdfield1} {\rm Let $(M,\gamma)$ be a pseudo-Riemannian manifold. The Levi-Civita connection of $\gamma$ yields a horizontal bundle $\mathcal{H}_\gamma$ defined by (\ref{Gammahoriz}) and $(\mathcal{H}_\gamma,\mathfrak{g}_{\mathcal{V}})$, where $\mathfrak{g}_{\mathcal{V}}$ is a vertical, compatible metric is a double field over $M$. In particular, there exists a double field with the component $\sigma$ equal to the transfer of $\gamma$ to its isomorphic bundle $p^{-1}(TM)$ and the form component $\psi=0$. The corresponding vertical metric is the vertical component of the Sasaki metric associated to $\gamma$ and its Levi-Civita connection.}\end{example}
\begin{example}\label{exdfield2} {\rm  Let $\mathcal{L}$ be a regular Lagrangian on $TM$ and let $\mathcal{H}_{\mathcal{L}},\sigma_{\mathcal{L}}$ be the horizontal bundle and the horizontal component of the metric $\mathfrak{g}_{\mathcal{L}}$ defined in Example \ref{exSasakitype} (transferred to $p^{-1}(TM)$). Furthermore, consider the $1$-form $\vartheta=d\mathcal{L}\circ S$ and the corresponding $2$-form $d'\vartheta$, where $d'$ is defined by the decomposition (\ref{descd}) associated to the horizontal bundle $\mathcal{H}_{\mathcal{L}}$. $d'\vartheta$ transfers to a $2$-form $\psi_{\mathcal{L}}$ on the bundle $p^{-1}(TM)$ and the triple $(\mathcal{H}_{\mathcal{L}},\sigma_{\mathcal{L}},\psi_{\mathcal{L}})$ defines a double field over $M$ that is canonically associated to the Lagrangian $\mathcal{L}$.}\end{example}

The use of double fields in physics requires an {\it action functional} deduced from physics requirements, particularly $T$-duality, which we do not discuss here. Instead, we will define an action that extends the one constructed for para-Hermitian manifolds in \cite{VT}, which was motivated by the geometry of the double fields theory of physics.
\begin{defin}\label{doublec} {\rm A {\it vertical, double metric connection} is a connection on the vertical bundle $\mathcal{V}$ that preserves the para-Hermitian metric $g$ and the double field metric $\mathfrak{g}_{\mathcal{V}}$.}\end{defin}
\begin{prop}\label{propdmc} A connection $\nabla^{\mathcal{V}}$ on $\mathcal{V}$ is a vertical double metric connection iff	 it preserves the subbundles $\mathcal{U}_\pm$ and their metrics $(pr_{\mathcal{V}_1})^*\sigma$.\end{prop}
\begin{proof} By (\ref{phi01}), the preservation of the pair of metrics $(g,\mathfrak{g}_{\mathcal{V}})$ is equivalent with the preservation of the pair $(\mathfrak{g}_{\mathcal{V}},\phi)$. Therefore, the eigenbundles $\mathcal{U}_\pm$ and the metrics $\mathfrak{g}_{\mathcal{V}}|_{\mathcal{U}_\pm}$ must be preserved too.\end{proof}
\begin{corol}\label{coroldmc} The connection $\nabla^{\mathcal{V}}$ is double metric iff it is expressible as
$$ \nabla^{\mathcal{V}}_{\mathfrak{Z}}(\iota_\pm\mathfrak{Y})
=\iota_\pm(\mathcal{D}^{\pm}_{\mathfrak{Z}}\mathfrak{Y})
\hspace{2mm} (\mathfrak{Y}\in\Gamma\mathcal{V}_1,\mathfrak{Z}\in
\chi(\mathfrak{T}M)),$$
where $\mathcal{D}^{\pm}$ is a pair of $\sigma$-metric connections on $\mathcal{V}_1\approx p^{-1}(TM)$.\end{corol}
\begin{proof} Obvious.\end{proof}

The action of a double field over $M$ will be defined by means of a canonical double metric connection obtained by adapting the construction given in \cite{VT}.
\begin{prop}\label{prop1end} There exist invariant procedures to derive double metric connections from the horizontal bundle $\mathcal{H}$ and the components $(\sigma,\psi)$ of a given double field.\end{prop}
\begin{proof} To prove the assertion, we give below one such procedure. We start by transferring the component metric $\sigma$ of the field to $\mathcal{V}_2\approx\mathcal{V}_1^*$ and to the horizontal bundle $\mathcal{H}\approx\mathcal{V}_1$. By adding up the results we get a metric $\mathfrak{g}_\sigma$ on $\mathfrak{T}M$ and there exists a corresponding connection $\bar{\nabla}^{D^\sigma}$ given by (\ref{nablabarD}), where $D^\sigma$ is the Levi-Civita connection of the metric $\mathfrak{g}_\sigma$. The restriction $D'^0_{\mathfrak{Z}}=\bar{\nabla}^{D^\sigma}_{\mathfrak{Z}} |_{\mathcal{V}_1}$ ($\mathfrak{Z}\in\chi(\mathfrak{T}M)$) yields a connection on $\mathcal{V}_1$, which, for $\mathfrak{Z}\in\mathcal{V}_1$, is just the Levi-Civita connection of $\sigma$ along the leaves tangent to $\mathcal{V}_1$. The invariant transformation
\begin{equation}\label{gotometricc}
D^0_{\mathfrak{Z}}\mathfrak{Y}_1=
D'^0_{\mathfrak{Z}}\mathfrak{Y}_1
+\frac{1}{2}
\sharp_\sigma\flat_{D'^0_{\mathfrak{Z}}\sigma} \mathfrak{Y}_1,\end{equation} where $D'^0_{\mathfrak{Z}}\sigma\in\Gamma\odot^2\mathcal{V}^*_1$,
yields a connection that preserves the metric $\sigma$ along any path (check by a calculation).
Then, following	 Corollary \ref{coroldmc}, we get a double metric connection $\mathcal{D}^0$ on $\mathcal{V}$ that corresponds to the pair $(D^0,D^0)$.

Now, we add a torsion defined by the field component $\psi$ and define a pair of connections
on $\mathcal{V}_1$ by
$$ D'^\pm_{\mathfrak{Z}}\mathfrak{Y}_1=
D^0_{\mathfrak{Z}}\mathfrak{Y}_1\pm\frac{1}{2}
\sharp_\sigma[i(\mathfrak{Y}_1) i(pr_{\mathcal{V}_1}\mathfrak{Z})d_{\mathcal{V}_1}\psi]
\hspace{2mm}(\mathfrak{Y}_1\in\mathcal{V}_1),$$
where $d_{\mathcal{V}_1}$ is the exterior differential along the leaves of $\mathcal{V}_1$.
The additional term is a vector in $\mathcal{V}_1$ and vanishes for $\mathfrak{Z}\in\mathcal{H},\mathcal{V}_2$. $D'^\pm$ still preserve $\sigma$ under parallel translations along paths in the leaves of $\mathcal{V}_1$. Then, we use	 transformations (\ref{gotometricc}) to change $D'^\pm$ to $\sigma$-preserving connections $D^\pm$ and introduce the double metric connection $\mathcal{D}$ on $\mathcal{V}$, which corresponds to the pair $D^\pm$ by Corollary \ref{coroldmc}.\end{proof}

To follow the footsteps of physical double field theory, we have to use a connection that is related to the $C$-bracket of string theory literature. We will use the {\it metric bracket} on $\mathcal{V}$ used in \cite{{VJMP},{VT}}.
\begin{defin}\label{defmetricbr} {\rm
Let $\mathfrak{g}_{\mathcal{V}}$ be a vertical metric and $\nabla$ a $\mathfrak{g}_{\mathcal{V}}$-metric connection on $\mathcal{V}$. The {\it operation $\wedge_{\nabla}$} is defined by the formula $$\mathfrak{g}_{\mathcal{V}}(\mathfrak{Y},\mathfrak{Y}' \wedge_{\nabla}\mathfrak{Y}'')=\frac{1}{2}
[\mathfrak{g}_{\mathcal{V}}(\mathfrak{Y}', \nabla_{\mathfrak{Y}}\mathfrak{Y}'') -\mathfrak{g}_{\mathcal{V}}(\mathfrak{Y}'', \nabla_{\mathfrak{Y}}\mathfrak{Y}')],$$
where all the arguments are in $\mathcal{V}$.
The {\it$\mathfrak{g}_{\mathcal{V}}$-metric bracket} is the bracket defined by the formula
$$[\mathfrak{Y}',\mathfrak{Y}'']_{\mathfrak{g}_{\mathcal{V}}}=
\mathcal{D}^0_{\mathfrak{Y}'}\mathfrak{Y}''
-\mathcal{D}^0_{\mathfrak{Y}''}\mathfrak{Y}'  -\mathfrak{Y}'\wedge_{\mathcal{D}^0}\mathfrak{Y}''.$$}\end{defin}

The name ``metric bracket" comes from the following properties \cite{{VJMP},{VT}}
\begin{equation}\label{prmbr} \begin{array}{l}
\mathfrak{Y}(\mathfrak{g}_{\mathcal{V}}(\mathfrak{Y'},
\mathfrak{Y}''))
=\mathfrak{g}_{\mathcal{V}}([\mathfrak{Y},
\mathfrak{Y}']_{\mathfrak{g}_{\mathcal{V}}}
+\frac{1}{2}grad_{\mathfrak{g}_{\mathcal{V}}}
(\mathfrak{g}_{\mathcal{V}}(\mathfrak{Y},\mathfrak{Y}'), \mathfrak{Y}'')\vspace*{2mm}\\

+\mathfrak{g}_{\mathcal{V}}(\mathfrak{Y}',
[\mathfrak{Y},\mathfrak{Y}'']_{\mathfrak{g}_{\mathcal{V}}}
+\frac{1}{2}grad_{\mathfrak{g}_{\mathcal{V}}}
(\mathfrak{g}_{\mathcal{V}}(\mathfrak{Y},
\mathfrak{Y}''))\vspace*{2mm}\\

[\mathfrak{Y},f\mathfrak{Y}']_{\mathfrak{g}_{\mathcal{V}}}
=f[\mathfrak{Y}, \mathfrak{Y}']_{\mathfrak{g}_{\mathcal{V}}} +(\mathfrak{Y}f)\mathfrak{Y}'
-\frac{1}{2}\mathfrak{g}_{\mathcal{V}}(\mathfrak{Y},
\mathfrak{Y}')grad_{\mathfrak{g}_{\mathcal{V}}}\, f.\end {array}\end{equation}

The metric bracket leads to new invariants of $\mathfrak{g}_{\mathcal{V}}$-preserving connections $\nabla$ \cite{VT}.
\begin{defin}\label{deformed invariants} {\rm The bracket $$[\mathfrak{Y},\mathcal{Y}']^\nabla_{\mathfrak{g}_{\mathcal{V}}} =[\mathfrak{Y},\mathcal{Y}']_{\mathfrak{g}_{\mathcal{V}}}+ \mathfrak{Y}\wedge_\nabla\mathcal{Y}'$$ is the {\it deformed Lie bracket}. The vertical tensor field
$$T_\nabla^\sigma(\mathfrak{Y},\mathcal{Y}')= \nabla_{\mathfrak{Y}}\mathfrak{Y}'-\nabla_{\mathfrak{Y}'}\mathfrak{Y} -[\mathfrak{Y},\mathcal{Y}']^\nabla_{\mathfrak{g}_{\mathcal{V}}}$$ is the {\it deformed torsion} of $\nabla$. The vertical tensor field
$$\tau_\nabla(\mathfrak{Y},\mathcal{Y}',\mathfrak{Y}'')=
\mathfrak{g}_{\mathcal{V}}(T_\nabla^{\mathfrak{g}_{\mathcal{V}}}
(\mathfrak{Y},\mathcal{Y}'), \mathfrak{Y}'')$$ is the {\it Gualtieri torsion} of $\nabla$.}
\end{defin}
\begin{prop}\label{exprGtors} If $\nabla$ is a $\mathfrak{g}_{\mathcal{V}}$-metric connection and if we denote $$\Theta(\mathfrak{Y},\mathfrak{Y}')=\nabla_{\mathfrak{Y}}\mathfrak{Y}' -\mathcal{D}^0_{\mathfrak{Y}}\mathfrak{Y}',\;\;
\Xi(\mathfrak{Y},\mathfrak{Y}',\mathfrak{Y}'')=
\mathfrak{g}_{\mathcal{V}}(\Theta(\mathfrak{Y},\mathfrak{Y}'),
\mathfrak{Y}''),$$ we get
$$\begin{array}{l}
\Xi(\mathfrak{Y},\mathfrak{Y}',\mathfrak{Y}'')= -\Xi(\mathfrak{Y},\mathfrak{Y}'', \mathfrak{Y}'),\vspace*{2mm}\\
\tau_\nabla(\mathfrak{Y},\mathfrak{Y}',\mathfrak{Y}'')=
\sum_{Cycl(\mathfrak{Y},\mathfrak{Y}',\mathfrak{Y}'')}
\Xi(\mathfrak{Y},\mathfrak{Y}',\mathfrak{Y}'').\end{array}$$\end{prop}
\begin{proof} The first equality follows from the metric character of the two connections. The second follows from the definition of $\tau_\nabla$ by a straightforward calculation.\end{proof}
\begin{corol}\label{coroltau} The Gualtieri torsion $\tau_\nabla$ is totally skew symmetric.\end{corol} \begin{proof} Obvious.\end{proof}
\begin{prop}\label{prop2end} There exist invariant procedures to derive a double metric connection with a vanishing Gualtieri torsion starting from the data of a given double field.\end{prop}
\begin{proof} We may use the procedure given in \cite{VT}.
Change the connections $D^\pm$ by the transformation
$$\tilde{D}^\pm_{\mathfrak{X}}\mathfrak{Y} =D^\pm_{\mathfrak{X}}\mathfrak{Y}_1,\;
\tilde{D}^\pm_{\mathfrak{Y}}\mathfrak{Y}_1 =D^\pm_{\mathfrak{Y}}\mathfrak{Y}_1 +pr_{\mathcal{V}_1}
pr_{\mathcal{U}_\pm}[pr_{\mathcal{U}_\mp}\mathfrak{Y}, \iota_\pm\mathfrak{Y}_1]_\sigma,$$
where $\mathfrak{X}\in\mathcal{H},\mathfrak{Y}\in\mathcal{V}, \mathfrak{Y}_1\in\mathcal{V}_1$.
A technical calculation, namely the one made for formula (4.15) of \cite{VT}, which uses (\ref{prmbr}), shows that $\tilde{D}^\pm$ preserve the metric $\sigma$. Thus, the change provides a double metric connection $\tilde{\mathcal{D}}$ on $\mathcal{V}$ that corresponds to the pair $\tilde{D}^\pm$. Another technical calculation shows that the Gualtieri torsion of $\tilde{\mathcal{D}}$ vanishes if two arguments belong to $\mathcal{U}_\pm$ and the third to $\mathcal{U}_\mp$.
Now, denote by $\tilde{\Theta},\bar{\Theta}$ the values of the tensor $\Theta$ of Proposition \ref{exprGtors} for the connections $\tilde{\mathcal{D}},\bar{\mathcal{D}}$ where $\bar{\mathcal{D}}= \tilde{\mathcal{D}}+\Phi$ for an arbitrary tensor $\Phi$. Then, we get
$$\tau_{\bar{D}}(\mathfrak{Y},\mathfrak{Y}',\mathfrak{Y}'') =\tau_{\tilde{D}}(\mathfrak{Y},\mathfrak{Y}',\mathfrak{Y}'')+
\mathfrak{g}_{\mathcal{V}}(\Phi(\mathfrak{Y},\mathfrak{Y}'), \mathfrak{Y}'')$$ and there is a unique choice of $\Phi$ such that the second term in the right hand side is totally skew symmetric and $\tau_{\bar{D}}=0$, given by
$$\mathfrak{g}_{\mathcal{V}}(\Phi(\mathfrak{Y},\mathfrak{Y}'), \mathfrak{Y}'')=-\frac{1}{3}\tau_{\tilde{D}}
(\mathfrak{Y},\mathfrak{Y}',\mathfrak{Y}'').$$
We regard $\bar{\mathcal{D}}$ as the connection required by the proposition and call it the {\it field-adapted connection}.
\end{proof}

Since $\bar{\mathcal{D}}$-covariant derivatives in vertical directions work like on para-Hermitian manifolds, we can transfer the definition of the action given in \cite{VT}, thereby, keeping close to double field theory.
\begin{defin}\label{curbdeform} {\rm The vertical tensor field
$$R^{\mathfrak{g}_{\mathcal{V}}}_{\bar{\mathcal{D}}}
(\mathfrak{Y},\mathfrak{Y}')\mathfrak{Y}''=
\bar{\mathcal{D}}_{\mathfrak{Y}}
\bar{\mathcal{D}}_{\mathfrak{Y}'}\mathfrak{Y}''
-\bar{\mathcal{D}}_{
\mathfrak{Y}}\bar{\mathcal{D}}_{\mathfrak{Y}'}\mathfrak{Y}''
-\bar{\mathcal{D}}_{[\mathfrak{Y},\mathcal{Y}']^{\bar{\mathcal{D}}}_{
\mathfrak{g}_{\mathcal{V}}}}\mathfrak{Y}''$$ is the {\it deformed, vertical curvature} of the connection $\bar{\mathcal{D}}$.
The vertical tensor field
$$ \begin{array}{r}
Ric^{\mathfrak{g}_{\mathcal{V}}}_{\bar{\mathcal{D}}}(\mathfrak{Y}, \mathfrak{Y}') =
 \frac{1}{2}\sum_{i=1}^m[
<\theta^i,R^{\mathfrak{g}_{\mathcal{V}}}_{
\bar{\mathcal{D}}}(\frac{\partial}{\partial y^i},\mathfrak{Y})\mathfrak{Y}'>+
<\kappa_i,R^{\mathfrak{g}_{\mathcal{V}}}_{
\bar{\mathcal{D}}}(\frac{\partial}{\partial z_i},
\mathfrak{Y})\mathfrak{Y}'>\vspace*{2mm}\\ +
<\theta^i,R^{\mathfrak{g}_{\mathcal{V}}}_{
\bar{\mathcal{D}}}(\frac{\partial}{\partial y^i},\mathfrak{Y}')\mathfrak{Y}>+
<\kappa_i,R^{\mathfrak{g}_{\mathcal{V}}}_{
\bar{\mathcal{D}}}(\frac{\partial}{\partial z_i},
\mathfrak{Y}')\mathfrak{Y}>], \end{array}$$ where $\theta^i,\kappa_i$ are defined by (\ref{hcobase}),
is the {\it deformed, vertical, Ricci curvature} of the connection $\bar{\mathcal{D}}$.
The scalar
$$ \rho^{\mathfrak{g}_{\mathcal{V}}}_{
\bar{\mathcal{D}}}=\sum_{q,s=1}^{2m}
\mathfrak{g}_{\mathcal{V}}^{qs}
Ric^{\mathfrak{g}_{\mathcal{V}}}_{
\bar{\mathcal{D}}qs}, $$
where the tensor components are with respect to any local basis of the bundle $\mathcal{V}$, is the {\it deformed, vertical, scalar curvature} of the connection $\bar{\mathcal{D}}$.}\end{defin}

If the manifold $M$ is oriented and if $(x^i)$ are	positive local charts, the double field $(\mathcal{H},\mathfrak{g}_{\mathcal{V}})$ yields the volume form
$$ d(vol)= |det\,\sigma|^{1/2}dx^1\wedge...\wedge dx^m
\wedge \theta^1\wedge...\wedge \theta^m\wedge \kappa_1\wedge...\wedge \kappa_m$$
on $\mathfrak{T}M$ (the $dx^i$ and $\theta^i$ portions of the form multiply by $J^{-1}$ and the $\kappa_i$ portion multiplies by $J$ under a
coordinate transformation $\tilde{x}^i=\tilde{x}^i(x^j)$ with the Jacobian $J$). Now, we can formulate the final result as
\begin{prop}\label{ultima} Let $(\mathcal{H},\mathfrak{g}_{\mathcal{V}})$ be a double field over the manifold $M$ with {\it density} $\varphi\in C^\infty(\mathfrak{T}M)$. Then, provided that the integral is finite, the formula
$$ \mathcal{A}(\mathfrak{g}_{\mathcal{V}},\mathcal{H})
=\int_{\mathfrak{T}M}e^{-2\varphi}
\rho^{\mathcal{V}}_{\mathfrak{g}_{\mathcal{V}}}d(vol)$$
yields a well defined action functional of the field.\end{prop}

Of course, any possible applications of this result should be decided by physics.

{\small Department of Mathematics, University of Haifa, Israel. E-mail: vaisman@math.haifa.ac.il
\end{document}